%% file: Liu_WB_PAMPA_Tri.tex
\documentclass[review,hidelinks,onefignum,onetabnum]{siamart220329}

\input{ex_shared}

\ifpdf
\hypersetup{
  pdftitle={Well-balanced Point-Average-Moment PolynomiAl-interpreted (PAMPA) Methods for Shallow Water Equations with Coriolis forces on Triangular Meshes},
  pdfauthor={Y. Liu}
}
\fi

\begin{document}

\maketitle

\nolinenumbers

\begin{abstract}
In this paper, we develop a novel well-balanced Point-Average-Moment PolynomiAl-interpreted (PAMPA) numerical method for solving the two-dimensional shallow water equations \vla{with temperature gradients} on unstructured triangular meshes. The proposed PAMPA method use a globally continuous representation of the variables, with degree of freedoms (DoFs) consisting of point values on the edges and average values within each triangular element. The update of cell averages is carried out using a conservative form of the partial differential equations (PDEs), while \vla{the update of point values---unconstrained by local conservation---follows a non-conservative formulation. The powerful PAMPA framework offers great flexibility in the choice of variables for the non-conservative form, including conservative variables, primitive variables, and other possible sets of variables. In order to preserve a wider class of steady-state solutions, we introduce pressure-momentum-temperature variables instead of using the standard conservative or primitive ones. By utilizing these new variables and the associated non-conservative form, along with adopting suitable Gaussian quadrature rules in the discretization of conservative form, we prove that this new class of schemes is well-balanced for both ``lake at rest'' and isobaric steady states. We validate the performance of the proposed well-balanced PAMPA method through a series of numerical experiments, demonstrating their high-order accuracy, well-balancedness, and robustness.} 
\end{abstract}



\section{Introduction}
This paper focus on introducing a novel well-balanced (WB) Point-Average-Moment PolynomiAl-interpreted (PAMPA) method to solve the two-dimensional (2-D) shallow water equations \vla{with horizontal temperature gradients} on general unstructured triangular meshes \vla{(see, e.g., \cite{Dellar03,Ripa,Ripa95})}:\vla{ 
\begin{equation}\label{1.1}
  \frac{\partial \mbf u}{\partial t}+{\rm div}~\mbf f(\mbf u)=\mbf S,
\end{equation}
where $t$ denotes time, and
\begin{equation}\label{1.1a}
 \mbf u=\begin{pmatrix}h\\h\bm\nu\\h\theta\end{pmatrix}, \quad \mbf f(\mbf u)=\begin{pmatrix}h\bm\nu\\h\bm\nu\bigotimes\bm\nu+\frac{g}{2}h^2\theta{\rm Id}_d\\h\bm\nu\theta\end{pmatrix},\quad \mbf S=\begin{pmatrix}0\\-gh\theta\nabla Z\\0\end{pmatrix}
\end{equation}
represent the vector of conservative variables, flux tensor, and source term, respectively. In \eqref{1.1a}, $h$ represents the water depth, $\bm\nu=(u,v)^\top$ is the velocity field, ${\rm Id}_d$ is the $d\times d$ identity matrix (with $d=2$ in this work), $Z$ is the bottom topography, $g$ is the gravity constant, and $\theta$ denotes the potential temperature field.  The system \eqref{1.1}--\eqref{1.1a}, referred to as the Ripa model, comprises the shallow water equations and terms which account for horizontal temperature fluctuations, making it an important role in modeling ocean currents and understanding real-world phenomena. The governing equations are derived based on multilayered ocean models, and vertically integrating the velocity fields, density, and horizontal pressure gradients in each layer. The inclusion of temperature is particularly advantageous since the temperature-driven forces acting on the water influence the motion and behaviour of the ocean currents.} 

The system \eqref{1.1}--\eqref{1.1a} is a hyperbolic system of balance laws \vla{provided it admits real eigenvalues and linearly independent eigenvectors}. The development of robust and accurate numerical methods for the computation of its solutions is both important and challenging as this nonlinear hyperbolic system admits both smooth and nonsmooth solutions: shocks, rarefaction waves, and their interactions. \vla{Additionally, \eqref{1.1} admits non-trivial steady-state solutions, which often respect a delicate balance between the flux and the source terms at the PDE level. These steady states are of particular interest, as many physically relevant waves can be viewed as small perturbations of such equilibria. Standard methods often treat flux derivatives and source terms differently, leading to circumstances where solutions that are stationary at the PDE level fail to remain stationary in numerical simulations. In such cases, small perturbations of steady states can be amplified by numerical scheme, generating unphysical waves with magnitudes proportional to the grid size. These spurious waves may overshadow the physical waves of interest, causing a phenomenon known as the ``numerical storm'' (see, e.g., \cite{KLshw,LeV98,NPPN_Sw}). While grid refinement can mitigate this issue, extreme refinement is often needed until spurious waves are significantly smaller than the physical waves that one wants to resolve in a near-equilibrium setup. However, the increase in computational cost is usually unaffordable. To address these challenges, the developed numerical method should be well-balanced, meaning it is capable of exactly balancing the flux and source terms to preserve (some) relevant steady states within machine accuracy and effectively capture the nearly steady-state flows on relatively coarse meshes. Furthermore, a good numerical scheme for the Ripa model should possess another important property: it should not develop spurious pressure oscillations in the neighborhood of temperature jumps between, for instance, the ``warm'' and ``cold'' water. These oscillations are analogous to those observed at material interfaces in compressible multifluid computations (see, e.g., \cite{Abgrall_2001_multif} and references therein.}

\vla{Note that if the potential temperature field is excluded by taking $\theta\equiv 1$, the system \eqref{1.1}--\eqref{1.1a} reduces to the classical Saint--Venant system of shallow water equations \cite{SV}, which has been extensively studied and applied to model water flows in rivers, lakes, and coastal areas. One of the most important steady-state solution is} the ``lake at rest'' equilibrium:
\begin{equation}\label{1.2}
 \bm\nu\equiv\mbf 0,\quad w:=h+Z\equiv{\rm Const},
\end{equation} 
where $w$ denotes the water surface level. In the past two decades, various WB methods, in the sense that it exactly preserves ``lake at rest'' equilibria \eqref{1.2}, for numerically solving the two-dimensional shallow water equations have been proposed. \vla{These methods include finite difference (FD) methods \cite{Wang_2022_FDSWE, Ren2024, Zhang_2023_FDSWE, Li_2020_FDSWE, Xing_2005_FDSWE}, finite volume (FV) methods \cite{Li_2012_FVSWE,NPPN_Sw,NXS,KPshw,BEKP,BMK2016,CMOT_WBSW,FMT11, Kurganov2022,Yan2024, DelGrosso2024}, discontinuous Galerkin (DG) methods \cite{LLTX,Xing_2013WBSWE, Zhang_2021_DGSWE,Mantri2024}, and continuous finite element methods \cite{Azerad_2017_FEMSW, Knobloch_2024_FEM,Hajduk}. Additionally, several other contributions deserve mention. In the framework of residual distribution (RD), well-balanced simulations for two-dimensional shallow water equations were demonstrated in \cite{RB09}, where the authors adapted the RD technology to develop second-order schemes that yield results comparable to the ones given by state of the art finite volume discretizations on unstructured triangular meshes. In \cite{Hauck_2020_EGSWE}, an enrich Galerkin (EG) method was proposed by extending the classical finite element approximation space with discontinuous functions supported on elements. Similar to DG, the EG method guarantees local conservation of all primary unknowns while the total number of degrees of freedom is substantially lower than that for the DG space of the same order.}  

\vla{In the presence of the horizontal temperature gradients, however, the structure of the steady-state solutions becomes more complex. In this case, the equilibria of interest are not only the ``lake at rest'':
\begin{equation}\label{1.2a}
  \bm\nu=\mbf 0,\quad \theta={\rm Const},\quad w=h+Z={\rm Const},
\end{equation}
but also the isobaric equilibria:
\begin{equation}\label{1.2b}
  \bm\nu=\mbf 0,\quad Z={\rm Const},\quad p:=h^2\theta={\rm Const},
\end{equation}
where the variable $p$ is referred to as the ``temperature dependent'' pressure. Several attempts have been made in the literature to design well-balanced schemes for the two-dimensional Ripa model \eqref{1.1}--\eqref{1.1a}. The first such work appears to be \cite{CKL}, where a well-balanced positivity-preserving central-upwind scheme coupled with an interface tracking method was developed to preserve both ``lake at rest'' and isobaric equilibria. Other examples, which are capable of preserving only the ``lake at rest'' equilibrium, include FD methods \cite{HL,Ren_Ripa_2024,Liu_2025_Ripa}, FV methods \cite{TouKli,Jelti_2025Ripa,KLZ_2D}, and DG methods \cite{Qian_2018_DGRipa,Li_2020_DGRIPA,Huang_2023}. The aforementioned works assume that the approximate solution is discontinuous on the interior edges of the mesh.}

\vla{Recently, a new class of schemes with globally continuous approximations, inspired by the so-called Active Flux method (see, e.g., \cite{ER_AF2, Eyman13, ER_AF1, Eymann_13,HKS,Barsukow_AF}), was introduced in \cite{Abgrall_camc}. This novel approach, referred to as Active Flux-like method, employs degrees of freedom (DoFs) from quadratic polynomials---all lying on element boundaries---along with the traditional average values within each element. These DoFs are evolved using two formulations of the same PDEs: a conservative version and a non-conservative version. It has been demonstrated that under mild assumptions on the scheme, and assumptions similar to those of the Lax--Wendroff theorem, the solution will converge to a weak solution of the studied problem. The method of lines formulation simplifies time stepping, making it more efficient and applicable for nonlinear hyperbolic problems. Formal third-order accuracy is achieved on a compact computational stencil. This contrasts with FD and FV methods, which typically improve accuracy by extending the computational stencil and using higher-order reconstructions, such as Essentially Non-Oscillatory (ENO) and Weighted ENO (WENO) schemes. Comparing with DG methods, Active Flux and Active Flux-like method are also advantageous, as its
shared DoFs requires less memory and less computational cost. Several extensions of this approach have been developed. For instance, by incorporating equilibrium variables into the non-conservative formulation expressed in primitive variables, a fully well-balanced Active Flux-like scheme was proposed for the one-dimensional shallow water equations in \cite{AbgrallLiu}. Additionally, by introducing higher moments of conservative variables as additional DoFs, the method was extended to arbitrarily high-order accuracy for one-dimensional hyperbolic conservation laws in \cite{AB_FE_FV} and hyperbolic balance laws in \cite{LiuBarsukow}, where the resulting high-order scheme is termed a hybrid finite element--finite volume method. Further extensions to multidimensional hyperbolic conservation laws have also been made, with both conservative and non-conservative formulations expressed in terms of conservative variables. These include the work in \cite{Abgrall_Lin_Liu}, which uses second- and third-order simplex elements on unstructured triangular meshes, and \cite{Abgrall_2025_AFC}, which employs Cartesian grids and refers to the resulting scheme as the generalized Active Flux scheme.

In this work, we use the framework in \cite{Abgrall_Lin_Liu} to develop well-balanced Active Flux-like methods, capable of exactly preserving both the ``lake at rest'' \eqref{1.2a} and isobaric \eqref{1.2b} equilibria, for Ripa system \eqref{1.1}--\eqref{1.1a}. 
We refer to this new scheme as well-balanced PAMPA (Point-Average-Moment PolynomiAl-interpreted), which serve as a novel extension of Active Flux-like method in solving multidimensional hyperbolic balance laws. We introduce only a third-order WB PAMPA scheme here and plan to extend the construction to higher-order schemes via higher-moments in future works. For the third-order case, the updates for all DoFs, including cell averages and point values, are formulated in a semi-discrete form and advanced in time using standard Runge--Kutta methods. More precisely, a conservative version of the studied PDEs is used to update the average values, while a non-conservative formulation, expressed in pressure-momentum-temperature variables, is employed to evolve the point values. The novelties and contributions of this work include the following:
\begin{itemize}
  \item Great flexibility in choosing various non-conservative formulations. Unlike \cite{Abgrall_Lin_Liu,Abgrall_2025_AFC}, where the non-conservative formulation is expressed solely in conservative variables, we introduce the pressure-momentum-temperature variables $(p,h\bm\nu,\theta)$ to rewrite the non-conservative formulation. This rewriting is the key to fulfilling both equilibria \eqref{1.2a} and \eqref{1.2b} simultaneously. We remark that using the conservative or primitive variables in the non-conservative form can preserve only the ``lake at rest'' equilibrium \eqref{1.2a}; see \cref{remark:pritive}.
  \item Well-balanced evolution is ensured in both conservative and non-conservative formulations. In order to achieve the WB update of cell averages, we project the bottom function into the same finite element space as the conservative variables and determin suitable Gauss quadrature rules in the computations of integrals to guarantee that all the integrals are properly calculated at the equilibria \eqref{1.2a} and \eqref{1.2b}. For the WB discretization of non-conservative form, we use the introduced water surface level to decompose the source terms. These ensures that when a linear finite difference operator is applied to approximate all the spatial derivatives, the WB property is achieved. We prove the well-balanced property for both \eqref{1.2a} and \eqref{1.2b} in \cref{prop:WB}.
  \item Positivity preserving and oscillation-eliminating near strong discontinuities. Similar to the approach in \cite{Abgrall_Lin_Liu}, we implement a simple a-posteriori limiting mechanism based on the MOOD paradigm from \cite{CDL,Vilar} coupled with first-order schemes to guarantee the non-negativity of the water depth and potential temperature field. 
  \item Extensive numerical validation. We validate the high-order accuracy, well-balancedness, and robustness of the proposed schemes through a series of numerical experiments. 
\end{itemize}}

%

The rest of the paper is organized as follows. In \cref{sec2}, we introduce the novel third-order WB PAMPA methods. The low-order schemes, which will be used to retain the nonlinear stability, are described in \cref{sec3}. In \cref{sec4}, some numerical examples are presented to illustrate the high-order accuracy in smooth regions for general solutions, the good performance in simulating vortex propagation over flat and non-flat bottom topography, the WB property, and the robustness in capturing shocks and other complex wave patterns, of the proposed schemes. Finally, concluding remarks are given in \cref{sec5}.

\section{Well-balanced third-order PAMPA schemes}\label{sec2}
In this section, we introduce the new third-order semi-discrete well-balanced PAMPA numerical schemes for solving the Ripa system on triangular meshes. 

\vla{In the PAMPA (or Active Flux) framework, the DoFs consist of averages evolved using the conservative formulation \eqref{1.1}--\eqref{1.1a}, and point values updated by a non-conservative formulation (equivalent in smooth regions) of \eqref{1.1}--\eqref{1.1a}:
\begin{equation}\label{2.1}
  \frac{\partial \mbf v}{\partial t}+\mbf J\cdot\nabla\mbf v=\widetilde{\mbf S},\quad \mbf J\cdot\nabla\mbf v=A\frac{\partial\mbf v}{\partial x}+B\frac{\partial\mbf v}{\partial y},
\end{equation}
where $\mbf v=\Psi(\mbf u)$ with $\Psi$ being a one-to-one, continuously differentiable, and invertible mapping. The Jacobian matrix vector $\mbf J$ is defined as:
\begin{equation*}
  \mbf J=\frac{\partial \Psi(\mbf u)}{\partial \mbf u}\frac{\partial \mbf f(\mbf u)}{\partial\mbf u}\Big(\frac{\partial \Psi(\mbf u)}{\partial \mbf u}\Big)^{-1}.
\end{equation*}
There are many choices for the variables $\mbf v$. In order to preserve simultaneously the ``lake at rest'' \eqref{1.2a} and isobaric \eqref{1.2b} equilibria, instead of using the standard conservative variables as in \cite{Abgrall_Lin_Liu,Abgrall_2025_AFC} or primitive variables, we choose the pressure-momentum-temperature variables $(p,h\bm\nu,\theta)$ with its associated Jacobian matrices:
\begin{equation}\label{2.2a}
  A=\left(
      \begin{array}{cccc}
        0 & 2h\theta & 0 & h^2u \\
        \frac{1}{2}\big(g-\frac{u^2}{h\theta}\big) & 2u & 0 & \frac{hu^2}{2\theta} \\
        -\frac{uv}{2h\theta} & v & u & \frac{huv}{2\theta} \\
        0 & 0 & 0 & u \\
      \end{array}
    \right),\quad B=\left(
      \begin{array}{cccc}
        0 & 0 & 2h\theta & h^2v \\
        -\frac{uv}{2h\theta} & v & u & \frac{huv}{2\theta} \\
        \frac{1}{2}\big(g-\frac{v^2}{h\theta}\big) & 0 & 2v & \frac{hv^2}{2\theta} \\
        0 & 0 & 0 & v \\
      \end{array}
    \right).
\end{equation}
The system is hyperbolic, i.e., for any vector $\bm n=(n_x,n_y)^\top$, the matrix
\begin{equation*}
  \bm J\cdot\bm n=An_x+Bn_y
\end{equation*}
is diagonalizable in $\mathbb{R}$.

Finally, in order to exactly preserve both the ``lake at rest'' and isobaric equilibria using the above two non-conservative formulations, we follow the idea in \cite{Xing_2005_FDSWE} and split the source terms $-gh\theta\nabla Z$ in \eqref{1.1a} into two terms $\frac{g\theta}{2}\nabla (Z^2)-gw\theta\nabla Z$, and then $\widetilde{\mbf S}$ in \eqref{2.1} reads as
\begin{equation}\label{Snew}
  \widetilde{\mbf S}=\begin{pmatrix}0\\\frac{g\theta}{2}\nabla (Z^2)-gw\theta\nabla Z\\0\end{pmatrix}.
\end{equation}
As we will see below, this special splitting of the source term is crucial for the design of our high order WB schemes to the non-conservative formulation.} 

\begin{rmk}\label{rmk1}
\bla{Note that the conservative variable $\mbf u$ belongs to the following invariant domain:
\begin{equation*}
 \mathcal{D}_{\mbf u}=\{\mbf u=(h,h\bm\nu,h\theta)^\top\in\mathbb{R}^4: h>0, \theta>0\}. 
\end{equation*}
It is easy to find a mapping $\Psi$, which is one-to-one, continuously differentiable, and invertible, such that the two sets of variables used in the non-conservative formulations can be transformed from the conservative ones, i.e., $\mbf v=\Psi(\mbf u)$. For instance, we set $\mbf v=(p,h\bm\nu,\theta)$ with $p=h\cdot h\theta$ and $\theta=\frac{h\theta}{h}$.}
\end{rmk}

\subsection{Approximation space}\label{sec21}
From now on, we assume that a triangulation $\mathcal{E}:=\bigcup\limits_j E_j$ of the computational domain, which consists of triangular cells $E_j$ of size $\vert E_j\vert$, is given. For each triangle $E$, the DoFs are defined at the three vertices $\sigma_i$ (for $i=1,2,3$), the three midpoints $\sigma_i$ (for $i=4,5,6$) on the three edges, and the cell average; see Figure \ref{p2_ele}. In this figure, we also define the scaled normal vectors associated with each DoF, which will be used later. For any polynomial DoF $\sigma\in\{\sigma_1,\sigma_2,\ldots,\sigma_6\}$ over the element $E$:
\begin{itemize}
  \item If $\sigma$ is a vertex, $\mbf n_\sigma$ is the inward normal vector to the edge opposite to $\sigma$. For example,  $\mbf n_1$. 
  \item If $\sigma$ is a midpoint, $\mbf n_\sigma$ is the outward normal vector to the edge on which $\sigma$ is sitting. For example, $\mbf n_4$. 
\end{itemize}
\begin{figure}[ht!]
\centerline{\includegraphics[trim=0.01cm 0.01cm 0.01cm 0.01cm,clip,width=4.5cm]{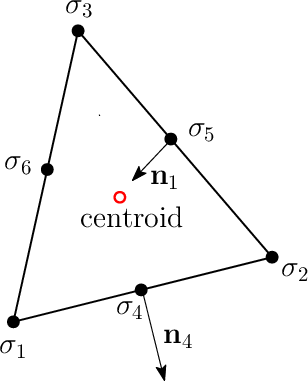}}
\caption{Quadratic triangular element.\label{p2_ele}}
\end{figure}

Equipped with the information of all DoFs for each triangle $E$, as shown in \cite[Section 2]{Abgrall_Lin_Liu}, we can construct a polynomial space $V$ that contains quadratic polynomials and also accommodates the cell average as an independent variable. The seven basis functions for this polynomial space, denoted by $\{\varphi_i\}_{i=1}^7$, are defined as:
\begin{equation*}
\begin{aligned}
  &\varphi_1=\lambda_1(2\lambda_1-1),\quad &&\varphi_2=\lambda_2(2\lambda_2-1),\quad &&\varphi_3=\lambda_3(2\lambda_3-1),\\
  &\varphi_4=4\lambda_1\lambda_2-20\lambda_1\lambda_2\lambda_3,\quad &&\varphi_5=4\lambda_2\lambda_3-20\lambda_1\lambda_2\lambda_3,\quad && \varphi_6=4\lambda_3\lambda_1-20\lambda_1\lambda_2\lambda_3,\\
  &\varphi_7=60\lambda_1\lambda_2\lambda_3,
\end{aligned}
\end{equation*} 
where $(\lambda_1,\lambda_2,\lambda_3)$ are the barycentric coordinates. One can verify that these basis functions satisfy the following conditions:
\begin{subequations}\label{basiscondi}
\begin{equation}\label{cond1}
  \varphi_i(\sigma_j)=\bla{\delta_{ij}},~\mbox{for}~1\leq i,j\leq6,\quad \int_E\varphi_i(\mbf x){\rm d}\mbf x=0,~\mbox{for}~i=1,\ldots,6
\end{equation}
and 
\begin{equation}\label{cond2}
  \varphi_7(\sigma_j)=0,~\mbox{for}~j=1,\ldots,6,\quad \int_E\varphi_7(\mbf x){\rm d}\mbf x=\bla{\vert E \vert}.
\end{equation}
\end{subequations}

Finally, in the third-order PAMPA method, we seek an approximation of the studied PDEs \eqref{1.1}--\eqref{1.1a}, denoted by \bla{$\mbf u_h$}\footnote{\bla{Following the finite element convention, we use the subscript $h$ to indicate a finite element approximation of a variable. For example, $\mbf u_h$ denotes a finite element approximation of $\mbf u$. Meanwhile, we have also used the letter $h$ for the water depth in this paper. Therefore, $h_h$ would stand for a finite element approximation of the water depth $h$.}}, which belongs to the finite dimensional space $V$ and within each element $E$ is given by 
\begin{equation}\label{2.5}
  \bla{\mbf u_h}|_E=\sum_{i=1}^{6}\mbf u_{\sigma_i}\varphi_i+\xbar{\mbf u}_E\varphi_7,
\end{equation}
where $\mbf u_{\sigma_i}$ is the point value of the variable $\mbf u$ at the element boundary DoF $\sigma_i$ and $\xbar{\mbf u}_E$ is the cell average defined as:
\begin{equation*}
  \xbar{\mbf u}_E=\frac{1}{|E|}\int_E\bla{\mbf u_h}(\mbf x){\rm d}\mbf x.
\end{equation*}
We project the bottom function $Z$ into the same space $V$, to obtain an approximation which is denoted by \bla{$Z_h$}.

\bla{
\begin{rmk}
We emphasize that the PAMPA methodology should not be confused with other approximation space enrichment approaches, such as enriched Galerkin (EG) methods and bubble function-based techniques. Compared to the quadratic polynomial space enriched by a cubic bubble function in \cite{Maeng_thesis}, we observe that the second condition in \eqref{cond1} does not hold in \cite{Maeng_thesis}, leading to a distinct approximation space. Furthermore, we also note that the idea of the EG method is to enhance the CG approximation space using element-local discontinuous functions and, by relying on a solution procedure nearly identical to that of the DG method (e.g., numerical edge fluxes and Riemann solvers), the PAMPA method employs a continuous solution reconstruction, eliminating the need for numerical fluxes and Riemann solvers. 
\end{rmk}}

\subsection{Semi-discrete scheme for the update of cell average}\label{sec22}
In this subsection, we describe how to evolve in time the cell average using the high-order spatial discretization.

A semi-discrete scheme for \eqref{1.1}--\eqref{1.1a} is a system of ODEs for the approximations of the cell average $\xbar{\mbf u}_E$ on a triangle $E$:
\begin{equation}\label{2.6}
  |E|\frac{{\rm d}\xbar{\mbf u}_E}{{\rm d}t}+\oint_{\partial E}{\mbf f}(\bla{\mbf u_h})\cdot\mbf n\; {\rm d}\gamma=\int_E\mbf S(\bla{\mbf u_h},\bla{Z_h})\;{\rm d}\mbf x,
\end{equation}
where $\mbf n$ is the outward unit normal at almost each point on the boundary $\partial E$ and we have applied the divergence theorem to rewrite the integral of the flux function over the triangle $E$.

\bla{Notice that we enforce global continuity on interior edges of the mesh: the point values on an edge are identical to those from the adjacent cells.} In the numerical simulations, \bla{we, therefore, can directly compute} the surface integral over $\partial E$ using a quadrature formula on each edge. Specifically,
\begin{equation}\label{2.7}
  \oint_{\partial E}{\mbf f}(\bla{\mbf u_h})\cdot\mbf n\; {\rm d}\gamma\bla{=}\sum_{l=1}^{3}\int_{{\rm ed}_l^E}{\mbf f}(\bla{\mbf u_h})\cdot\mbf n_l^E\; {\rm d}\gamma
\end{equation}  
where $\mbf n_l^E$ is the outer unit normal to the $l$-th edge of triangle $E$.

\bla{Generally speaking, the scheme described above preserves neither the ``lake at rest'' equilibrium nor the isobaric equilibrium unless suitable quadrature rules are specified. Recalling that we project the bottom topography into the same finite element space, the numerical solution $\mbf u_h$ and $Z_h$ are polynomials in each element $E$, making $\mbf f(\mbf u_h)$ and $\mbf S(\mbf u_h, Z_h)$ both polynomials as well.

In order to preserve the ``lake at rest'' equilibrium, we require exact computation of edge and area integrals at this steady state. This is achieved using Gaussian quadrature rules of degrees of precision at least four and five for edge and area integrals, respectively. However, these rules alone cannot preserve the isobaric equilibrium \eqref{1.2b}, since the interest quantity to preserve---the pressure---is a nonlinear function of the conservative variables. 

At the isobaric equilibrium, the bottom topography is flat, eliminating the need to modify the evaluation of the area integral. For the edge integral \eqref{2.7}, we note that the pressure is constant at the three DoFs on each element edge. This allows us to use a 3-point Gauss--Lobatto quadrature rule (collocation method) for edge integrals while preserving the isobaric equilibrium and numerical accuracy. 

To adaptively select the suitable quadrature rules for the edge integral \eqref{2.7}, we introduce a local plateau condition for the bottom topography within each element $E$. When the bottom topography is locally flat in an extended neighborhood of element $E$, we use the Gauss--Lobatto rule. Otherwise, we apply the Gauss--Legendre rule with 5 quadrature points. The extended neighborhood of $E$ is given by the index set $\bar{V}_E=\{\b{V}_{E_i}\}_{E_i\in\b{V}_{E_i}}$ (see Figure \ref{neighbour_cell}--(b), where $\bar{V}_E=\{\b{V}_{E},\b{V}_{E_1},\b{V}_{E_2},\b{V}_{E_3}\}$). Here, $\b{V}_{E_i}$ is the neighborhood index set and defined as the set containing $E_i$ and all elements sharing an edge with it (see Figure \ref{neighbour_cell}--(a), where $\b{V}_E=\{E,E_1,E_2,E_3\}$). An element $E$ satisfies the local plateau condition of bottom topography, if
\begin{equation}\label{plateau_Z}
  \max_{\mbf x\in\bar{V}_E}(Z(\mbf x))-\min_{\mbf x\in\bar{V}_E}(Z(\mbf x))\leq \varepsilon,
\end{equation}
where we set $\varepsilon=10^{-6}$ in the numerical experiments. 
\begin{figure}[ht!]
\centerline{\subfigure[Neighborhood cells of $E$]{\includegraphics[trim=0.01cm 0.01cm 0.01cm 0.01cm,clip,width=5.0cm]{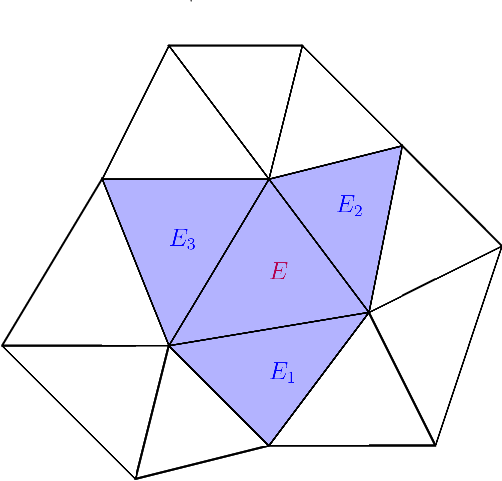}}\hspace*{0.5cm}
	\subfigure[Extended neighborhood cells of $E$]{\includegraphics[trim=0.01cm 0.01cm 0.01cm 0.01cm,clip,width=5.0cm]{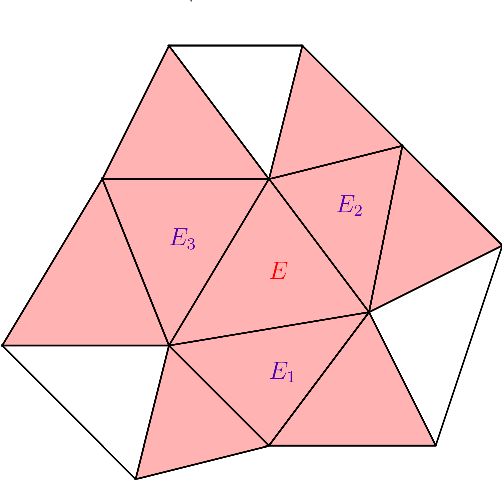}}}
\caption{\sf Mesh notation and the (extended) neighbourhood areas of an element $E$. \label{neighbour_cell}}
\end{figure}
The quadrature rule selection process is summarized in Algorithm \cref{alg:Gaussian_integrals}.
\begin{algorithm}
\color{blue}
\caption{Approximation of area integral in \eqref{2.6} and edge integral \eqref{2.7}}
\label{alg:Gaussian_integrals}
\begin{algorithmic}[1]
\FORALL{Element $E\in \mathcal{E}$ }
\STATE{Compute the area integral $\int_E\mbf S(\bla{\mbf u_h},\bla{Z_h})\;{\rm d}\mbf x$ using $7$-point Gaussian quadrature.}
\STATE{Check local plateau condition \eqref{plateau_Z} for $E$.}
\IF {$E$ is a locally flat}
\STATE{Compute edge integral \eqref{2.7} using 3-point Gauss--Lobatto rule (3 collocation points per edge).}
\ELSE
\STATE{Compute edge integral \eqref{2.7} using 5-point Gauss--Legendre rule}
\ENDIF
\ENDFOR
\end{algorithmic}
\end{algorithm}
With these adapted quadrature rules, we can demonstrate that the semi-discrete scheme for the conservative form is WB (see \cref{prop:WB}).} 
\bla{
\begin{rmk}\label{conservation}
It appears that other WB approximations of the source terms exist. For instance, one can decompose the source terms as 
\begin{equation*}
  -\int_{E}gh\theta\nabla Z\;{\rm d}\mbf x=\int_{E}\frac{g}{2}\nabla (h^2\theta)\;{\rm d}\mbf x-\int_{E}\frac{gh^2}{2}\nabla \theta\;{\rm d}\mbf x-\int_{E}gh\theta\nabla w\;{\rm d}\mbf x.
\end{equation*}
Note that the first integral can be cancelled by the acoustic part of the flux. The remaining two integrals, which naturally vanish at the ``lake at rest'' equilibrium, can be approximated using Gaussian quadrature rules with fewer quadrature points than those used in our approach.
However, in the case of flat bottom topography where the source terms vanish, this approximation is neither exactly zero nor consistent with zero. It is also important to note that the derivative terms $\nabla \theta$ and $\nabla w$ involve unknown functions $\theta$ and $h$. In this situation, conservation and convergence towards weak solutions may be problematic for discontinuous solutions, since the well-known Lax--Wendroff theorem. Similar observations in the FD WENO methods for the shallow water equations can be found in \cite{Xing_2005_FDSWE}.
\end{rmk}}

\subsection{Semi-discrete schemes for the update of point value}\label{sec23}  
In this subsection, we describe how to evolve the boundary DoFs of each triangle $E$ in time using the non-conservative formulations \eqref{2.1}--\eqref{2.2a}.


\bla{Since the update of point values is not constrained by local conservation and we are not restricted by a conservation formulation. Therefore, we can apply a finite difference-type method in the non-conservative form.} Following the approach introduced in \cite[Section 3.1]{Abgrall_Lin_Liu}, the evolution of any boundary DoF \vla{$\mbf v_\sigma$ ($\sigma\in\{\sigma_1,\sigma_2,\cdots,\sigma_6\}$)} is carried out by the following semi-discrete form:
\begin{equation}\label{2.12}
  \frac{{\rm d}\vla{\mbf v}_\sigma}{{\rm d}t}+\sum\limits_{E, \sigma\in E} \Phi_\sigma^E(\vla{\mbf v_h})=0,
\end{equation}
where $\mbf v_h$ is the numerical approximation of $\mbf v$.
\vla{Inspired by Residual Distribution schemes, and in particular the LDA \footnote{The LDA scheme is designed for steady advection problems on triangular meshes for scalar problems. Point value residuals are defined to update the solution similarly to our approach. It is locally conservative, linearity-preserving, and multidimensional upwind but does not suppress spurious oscillations at discontinuities. In its original version, it is second order accurate. The accuracy is obtained by defining the residual at this point as the product of the upwind parameter with $J\cdot\nabla u$. The upwind property is obtained by considering the sign of the scalar product of local velocity (hence $J$ for a scalar problem) against the normal opposite to the vertex where the residual is evaluated. Finally local conservation is guaranteed by multiplying the upwind parameter by a scaling factor, the so-called $N$ parameter, so that the sum of the fluctuations over the element is equal to $\int_E J\cdot \nabla u\;{\rm d}\mbf x$. The $J$ Jacobians are evaluated so that this integral is equal to the integral over the boundary of the normal flux, by some generalisation of the Roe average.} (Low Diffusion advection A) scheme \cite{Deconinck2007}, we can define point residual} $\Phi_\sigma^E(\vla{\mbf v_h})$ as
\begin{equation}\label{2.13}
 \bla{ \Phi_\sigma^E(\mbf v_h)=N_\sigma \big(K_\sigma^E\big)^+ \Big(\mbf J_\sigma\cdot \nabla \mbf v_\sigma-\widetilde{\mbf S}_\sigma\Big),\quad N_\sigma^{-1}= \sum\limits_{E, \sigma\in E} \big (K_\sigma^E\big )^+},
\end{equation}
\bla{where $\mbf J_\sigma=\mbf J(\mbf v_\sigma)$, $\nabla \mbf v_\sigma=\nabla \mbf v_h(\mbf v_\sigma)$, and $\widetilde{\mbf S}_\sigma=\widetilde{\mbf S}(\mbf v_\sigma, Z_\sigma)$. Since the problem is hyperbolic, for any vector $\mbf n_\sigma$ defined in \cref{sec21}, the matrix $K_\sigma^E$ can be decomposed into its positive and negative parts:
\begin{equation*}
  K^E_\sigma =\mbf J_\sigma\cdot \mbf n_\sigma = (K^E_\sigma)^-+(K^E_\sigma)^+=R\Lambda R^{-1}=R\Lambda^- R^{-1}+R\Lambda^+R^{-1},
\end{equation*}
where $R$ is the matrix of eigenvectors of $K^E_\sigma$ corresponding to eigenvalues $\{\lambda_1,\dots,\lambda_4\}$ (since $\mbf u\in\mathbb{R}^4$), and
\begin{equation*}
    \Lambda^+ = \diag\big(\max(\lambda_1, 0), \dots, \max(\lambda_4, 0)\big), ~
    \Lambda^-= \diag\big(\min(\lambda_1, 0), \dots, \min(\lambda_4, 0)\big).
\end{equation*}
}

\begin{rmk}
\vla{Recalling that the update of point values is primarily constrained only by the requirement that the resulting method remains stable. With this in mind, a natural approach is to incorporate an upwinding mechanism \footnote{\vla{Recently, we have proven that the central version also satisfies the stability requirement due to its staggered nature \cite{central_PAMPA}.}}, similar to the one used in the original LDA scheme \cite{Deconinck2007,Weide}.} \bla{Moreover, we can verify that this choice ensures
consistency in cases where $\nabla\mbf v_\sigma$ and $Z$ are constants. Namely, we recover
\begin{equation*}
  \sum\limits_{E, \sigma\in E} \Phi_\sigma^E(\mbf v_h)=\sum\limits_{E, \sigma\in E}N_\sigma \big(K_\sigma^E\big)^+\mbf J_\sigma\cdot \nabla \mbf v_\sigma=N_\sigma\Big(\sum\limits_{E, \sigma\in E}\big(K_\sigma^E\big)^+\Big)\mbf J_\sigma\cdot \nabla \mbf v_\sigma=J_\sigma\cdot \nabla \mbf v_\sigma.
\end{equation*}}
\end{rmk}

\vla{
To complete the computation in \eqref{2.13}, we also need to calculate $\nabla \mbf v_\sigma$ and $\widetilde{\mbf S}_\sigma$, which take the following forms:
\begin{equation*}
  \nabla \mbf v_\sigma=\begin{pmatrix}\nabla p_\sigma\\\nabla (h\bm\nu)_\sigma\\\nabla\theta_\sigma\end{pmatrix}\quad \mbox{and}\quad
  \widetilde{\mbf S}_\sigma\stackrel{\eqref{Snew}}{=}\begin{pmatrix}0\\\frac{g\theta_\sigma}{2}\nabla Z^2_\sigma-gw_\sigma\theta_\sigma\nabla Z_\sigma\\0\end{pmatrix},~~ w_\sigma=h_\sigma+Z_\sigma.
\end{equation*}
Obviously, we require high-order FD approximations for $\nabla \mbf v_\sigma$, $\nabla Z_\sigma$, and $\nabla Z^2_\sigma$. These can be obtained using the point values of the boundary DoFs and the point value at the centroid of each element. Since $Z$ is a given function, its point values, as well as those of $Z^2$, are immediately available. We now describe how to compute the point values $\mbf v_\sigma$ at the element boundaries and the point value at the centroid, denoted by $\mbf v_c$. First, using the polynomial expansion \eqref{2.5}, we compute the point value of $\mbf u$ at the centroid, denoted as $\mbf u_c$:
\begin{equation}\label{uc}
  \mbf u_c=\frac{20}{9}\xbar{\mbf u}_E-\frac{1}{9}\sum_{i=1}^3\mbf u_{\sigma_i}-\frac{8}{27}\sum_{i=4}^6\mbf u_{\sigma_i}.
\end{equation}  
Together with the boundary DoFs $\mbf u_\sigma$ and using \cref{rmk1}, we obtain the point values $\mbf v_\sigma=\Psi(\mbf u_\sigma)$ and $\mbf v_c=\Psi(\mbf u_c)$. Next, we employ these point values ($\{\mbf v_{\sigma_i}\}_{i=1}^6$ and $\mbf v_c$) to define a linear FD operator $D$, constructed using the Lagrange polynomial interpolation that includes $\mbf v_c$. The basis functions are defined as follows:
\begin{equation*}
\begin{aligned}
  &\phi_7=27\lambda_1\lambda_2\lambda_3,\\
  &\phi_1=\lambda_1(2\lambda_1-1)+\frac{1}{9}\phi_7, &&\phi_2=\lambda_2(2\lambda_2-1)+\frac{1}{9}\phi_7, &&\phi_3=\lambda_3(2\lambda_3-1)+\frac{1}{9}\phi_7,\\
  &\phi_4=4\lambda_1\lambda_2-\frac{4}{9}\phi_7, &&\phi_5=4\lambda_2\lambda_3-\frac{4}{9}\phi_7, && \phi_6=4\lambda_3\lambda_1-\frac{4}{9}\phi_7,
\end{aligned}
\end{equation*}
satisfying the interpolation condition $\phi_i(\sigma_j)=\delta_{ij}$ for $1\leq i,j\leq7$, where $\sigma_7$ represents the DoF at the centroid. Finally, we compute $\nabla \mbf v(\sigma)$ as
\begin{equation*}
  \nabla\mbf v\approx D(\mbf v)=\sum_{i=1}^{6}\mbf v_{\sigma_i}\nabla\phi_i+\mbf v_c\nabla\phi_7.
\end{equation*}
The same linear FD operator can be applied to $Z$ and $Z^2$, yielding
\begin{equation*}
  \nabla Z\approx D(Z)=\sum_{i=1}^{6}\mbf Z_{\sigma_i}\nabla\phi_i+\mbf Z_c\nabla\phi_7,\quad \nabla Z^2\approx D(Z^2)=\sum_{i=1}^{6}\mbf (Z_{\sigma_i})^2\nabla\phi_i+\mbf (Z_c)^2\nabla\phi_7.
\end{equation*}

\begin{rmk}
Alternatively, one could derive a high-order linear FD operator using only the boundary DoFs without incorporating the value at the centroid.  However, this introduces a hidden issue: the non-conservative form loses its connection to the conservative form. From the computation of the centroid point value $\mbf u_c$ in \eqref{uc}, we observe that this communication between the two formulations is maintained through the use of $\xbar{\mbf u}_E$.
\end{rmk}
}

\subsection{Summary of the third-order well-balanced schemes}\label{summary}
The proposed third-order WB PAMPA methods for the shallow water equations with temperature gradients are detailed in \cref{sec22} and \cref{sec23}. The evolution of the cell averages is governed by \eqref{2.6}, where the surface flux defined in \eqref{2.7} and the area integral of the source terms are evaluated by suitable Gaussian quadrature rules. The boundary point values are updated according to \eqref{2.1}--\eqref{2.2a} for the pressure-momentum-temperature variables $(p,h\bm\nu,\theta)$. The construction of the numerical schemes demonstrates that the PAMPA methods use a compact stencil and require fewer DoFs to achieve high order, offering potential advantages over WENO or DG methods. Moreover, the well-balanced property of the PAMPA methods is achieved without relying on complex numerical flux evaluations or hydrostatic reconstruction techniques. This advantage benefits from the continuous assumption of the solution and the flexible choice of the well-balanced evolution for the point values. Finally, by collecting the results of the previous subsections, it is straightforward to prove the following WB result:

\begin{prop}\label{prop:WB}
The third-order WB PAMPA methods for the shallow water equations with temperature gradients, as described above, \vla{maintain the WB property for the steady states ``lake at rest'' \eqref{1.2a} and isobaric \eqref{1.2b}}. That is, if the discrete data satisfy \vla{``lake at rest'':
\begin{equation}\label{2.16}
\begin{aligned}
  &(hu)_\sigma=(hv)_\sigma\equiv0,& &\theta_\sigma=\widehat{\theta}\equiv{\rm Const},& & h_\sigma+Z_\sigma=\widehat{w}\equiv{\rm Const},&& \forall \sigma,\\
  &\xbar{(hu)}_E=\xbar{(hv)}_E\equiv0,& &\xbar{(h\theta)}_E=\xbar{h}_E,& & \xbar{h}_E+\xbar{Z}_E=\widehat{w}\equiv{\rm Const},&& \forall E, 
\end{aligned}
\end{equation}
or isobaric:
\begin{equation}\label{2.16a}
\begin{aligned}
  &(hu)_\sigma=(hv)_\sigma\equiv0,& &Z_\sigma\equiv{\rm Const},& & p_\sigma=h_\sigma^2\theta_\sigma=\widehat{p}\equiv{\rm Const},&& \forall \sigma,\\
  &\xbar{(hu)}_E=\xbar{(hv)}_E\equiv0,& &\xbar{Z}_E\equiv{\rm Const},& & \xbar{p}_E=\xbar{h}_E\xbar{(h\theta)}_E=\widehat{p}\equiv{\rm Const},&& \forall E, 
\end{aligned}
\end{equation}}
then \eqref{2.6} and \eqref{2.12} reduce to
\begin{equation*}
  \frac{{\rm d}\xbar{\mbf u}_E}{{\rm d}t}=\mbf 0 \quad \mbox{and} \quad \frac{{\rm d}{\mbf v}_\sigma}{{\rm d}t}=\mbf 0.
\end{equation*}
\end{prop}
\begin{proof}
\bla{
We begin with the conservative formulation given in \eqref{2.6}. When the data satisfy \eqref{2.16}, we have 
\begin{equation*}
\begin{aligned}
  (h\bm\nu)_h&=\sum_{i=1}^{6}(h\bm\nu)_{\sigma_i}\varphi_i+\xbar{(h\bm\nu)}_E\varphi_7=\mbf 0,\\
  (h\theta)_h&=\sum_{i=1}^{6}(h\theta)_{\sigma_i}\varphi_i+\xbar{(h\theta)}_E\varphi_7=\widehat{\theta}h_h,\\
  h_h+Z_h&=\sum_{i=1}^{6}(h_{\sigma_i}+Z_{\sigma_i})\varphi_i+(\xbar{h}_E+\xbar{Z}_E)\varphi_7=\widehat{w}.
  \end{aligned}
\end{equation*}
and then we can verify that $\nabla\cdot\mbf f(\mbf u_h)=\mbf S(\mbf u_h, Z_h)$. Since both the edge and area integrals in \eqref{2.6} and \eqref{2.7} are calculated exactly for the still water using suitable Gaussian quadrature rules, the fluctuations of \eqref{2.6} equivalent to
\begin{equation*}
  \sum_{l=1}^{3}\int_{{\rm ed}_l^E}{\mbf f}(\mbf u_h)\cdot\mbf n_l^E\; {\rm d}\gamma-\int_E\mbf S(\mbf u_h,Z_h)\;{\rm d}\mbf x=\int_E \big(\nabla\cdot\mbf f(\mbf u_h)-\mbf S(\mbf u_h, Z_h)\big)\;{\rm d}\mbf x=0.
\end{equation*}
This directly implies that $\frac{{\rm d}\xbar{\bm u}_E}{{\rm d}t}=\mbf 0$. When the data satisfy \eqref{2.16a}, we have $p_{\sigma_i}=\widehat{p}$, $\nabla Z_h\equiv\mbf 0$. Therefore, for edge integral,
\begin{equation*}
  \sum_{l=1}^{3}\int_{{\rm ed}_l^E}{\mbf f}(\mbf u_h)\cdot\mbf n_l^E\; {\rm d}\gamma=\sum_{l=1}^{3}\sum_{i=1}^{3}
  \begin{pmatrix}
      0 \\
      \frac{g}{2}p_{\sigma_i} \\
      0 \\
    \end{pmatrix}
 \widetilde{\omega}_i\mbf n_l^E=\frac{g}{2}\widehat{p}\sum_{l=1}^{3}\mbf n_l^E=\mbf 0,
\end{equation*}
where $\{\widetilde{\omega}_i\}_{i=1}^3$ are the 3-point Gauss-Lobatto weights. For the area integral,
\begin{equation*}
  \int_E\mbf S(\mbf u_h,Z_h)\;{\rm d}\mbf x=-g\vert E\vert\sum_{\mbf x_{i}}
    \begin{pmatrix}
      0 \\
      (h\theta)_h(\mbf x_i)\nabla Z_h(\mbf x_i) \\
      0 \\
    \end{pmatrix}\omega_{i}=\mbf 0,
\end{equation*}
where $(\mbf x_{i},\omega_{i})$ are Gaussian quadrature pairs.

Next, we consider the non-conservative form \eqref{2.1}--\eqref{2.2a}. For the ``lake at rest'' equilibrium \eqref{2.16}, we obtain
\begin{equation*}
\begin{aligned}
  \mbf J_\sigma\cdot \nabla \mbf v_\sigma-\widetilde{\mbf S}_\sigma
  &=\left(\begin{array}{c}
     0 \\
     \frac{g}{2}D(p)-\frac{g\widehat{\theta}}{2}D(Z^2)+g\widehat{\theta}\widehat{w} D(Z) \\
     0
     \end{array}\right)\\
  &=\left(\begin{array}{c}
     0 \\
     D\Big(\frac{g\widehat{\theta}}{2}(h^2-Z^2)+g\widehat{\theta}\widehat{w}Z\Big) \\
     0
     \end{array}\right)
  =\left(\begin{array}{c}
     0 \\
     \frac{g\widehat{\theta}\widehat{w}}{2}D\big(\widehat{w}\big) \\
     0
     \end{array}\right)=\mbf 0,
  \end{aligned}
\end{equation*}
where the first equality follows from the fact that both $\theta$ and $w$ are constants at steady state; the second equality is due to the linearity of the FD operator $D$ and the fact that $\widehat{w}={\rm Const}$; the third equality is just a simple regrouping of terms inside the parenthesis; and the last equality follows from the fact that $\widehat{w}={\rm Const}$ and the consistency of the FD operator $D$. Similarly, for the isobaric equilibrium \eqref{2.16a}, we have
\begin{equation*}
  \mbf J_\sigma\cdot \nabla \mbf v_\sigma-\widetilde{\mbf S}_\sigma
  =\left(\begin{array}{c}
     0 \\
     \frac{g}{2}D(p) \\
     0
     \end{array}\right)=\mbf 0,
\end{equation*}
since $p=\widehat{p}={\rm Const}$ at steady state and the FD operator $D$ is consistent. From \eqref{2.13}, we conclude that $\Phi_\sigma^E(\mbf v)=0$ for both steady states, which implies that $\frac{{\rm d}{\mbf v}_\sigma}{{\rm d}t}=0$.}

\end{proof}

\bla{
\begin{rmk}\label{remark:pritive}
In the non-conservative form, if we consider the primitive variables $(h,\bm\nu,\theta)$ and the associated Jacobian matrices and source terms:
\begin{equation*}
   A=\left(
      \begin{array}{cccc}
        u & h & 0 & 0 \\
        g\theta & u & 0 & \frac{gh}{2} \\
        0 & 0 & u & 0 \\
        0 & 0 & 0 & u \\
      \end{array}
    \right),\quad B=\left(
      \begin{array}{cccc}
        v & 0 & h & 0 \\
        0 & v & 0 & 0 \\
        g\theta & 0 & v & \frac{gh}{2} \\
        0 & 0 & 0 & v \\
      \end{array}
    \right),\quad 
    \widetilde{\mbf S}=\left(
      \begin{array}{c}
        0 \\
        -g\theta \nabla Z \\
        0 \\
      \end{array}
    \right).
\end{equation*}
It is trivial to verify that this formulation preserves the ``lake at rest'' equilibrium since $\theta={\rm Const}$, $\bm\nu=\mbf 0$, and $g\theta\nabla(h+Z)=\mbf 0$. For the isobaric steady state, given that $\bm\nu=\mbf 0$, we only need to examine the momentum equations. Here, we obtain $g\theta_\sigma  D(h)+\frac{gh_\sigma}{2}D(\theta)\neq\frac{g}{2}D(\theta h^2)$, since $D$ is a linear FD operator. Consequently, the fluctuations for the momentum equations do not vanish at the isobaric equilibrium, implying that the scheme is not WB for this steady state. A similar analysis can be performed when using conservative variables.
\end{rmk}}
\section{Non-linear stabilization and low-order schemes}\label{sec3}
As pointed out in \cite{Abgrall_Lin_Liu}, the proposed high-order schemes possess linear stability. However, when the solution develops discontinuities, the high-order schemes will be prone to numerical oscillations. Additionally, the high-order schemes cannot guarantee the preservation of water depth and temperature filed positivity. \vla{These issues can lead to nonlinear instability and even cause code crashes. To address them,} we need some sort of nonlinear limiting. For this purpose, we employ a simplified version of the MOOD paradigm from \cite{CDL,Vilar} combined with low-order schemes to ensure the nonlinear stability. \bla{The idea of the MOOD paradigm operates by first applying a high-order numerical method over the entire domain for a single time step to obtain a candidate solution. Then, for each element, the behaviour of this solution is locally checked using certain, pre-selected admissibility criteria, such as computer admissibility, physical admissibility, monotonicity, and others. If the candidate solution meets the selected criteria, it is accepted. Otherwise, it is locally recomputed with a spatially lower order numerical method. This procedure is repeated until said criteria are universally satisfied, or when a robust first order scheme is employed. Essentially, this approach is a-posteriori limiting strategy. In our implementation, the MOOD paradigm involves working with two numerical schemes ranging from a third-order WB PAMPA method to a first-order robust Lax--Friedrichs scheme. The latter is designed to be positivity-preserving and oscillation-eliminating.} For more details on the MOOD approach and the selection of admissibility criteria, we refer to \cite[Section 3.4]{Abgrall_Lin_Liu}. In the following, we focus on describing the first-order schemes\bla{, which are not WB but should be monotone}.

\subsection{First-order scheme for cell average}
For the evolution of the cell average $\xbar{\bm u}_E$, we simply take a standard finite volume numerical flux (so we do not use the boundary values). For instance, we use the local Lax-Friedrichs numerical flux, replacing \eqref{2.7} with
\begin{equation}\label{3.1}
  \oint_{\partial E}{\mbf f}(\mbf u)\cdot\mbf n\; {\rm d}\gamma\approx\sum_{l=1}^{3}\ell_l\widehat{\mbf f}(\xbar{\mbf u}_E,\xbar{\mbf u}_{E'};\mbf n_l^E),
\end{equation}
where $\xbar{\mbf u}_{E'}$ represents the cell average of the neighboring triangle that shares the $l$-th edge with triangle $E$. The numerical flux $\widehat{\mbf f}$ is defined as
\begin{equation}\label{3.2}
  \widehat{\mbf f}(\mbf a_1,\mbf a_2;\mbf n)=\frac{1}{2}\Big[\mbf f(\mbf a_1)\cdot\mbf n+\mbf f(\mbf a_2)\cdot\mbf n-\alpha(\mbf a_2-\mbf a_1)\Big],
\end{equation}
where,
\begin{equation*}
  \alpha=\max_{\mbf a=\{\mbf a_1,\mbf a_2\}}\rho\big ( \mbf J(\mbf a)\cdot \mbf n\big ),
\end{equation*}
where $\rho(B)$ is the spectral radius of the matrix $B$. In addition to this modification of flux evaluation, we also replace the source terms approximation using the midpoint quadrature rule. That is, we compute 
\begin{equation*}
\begin{aligned}
 &-g\int_Eh\vla{\theta}\nabla Z_h\;{\rm d}\mbf x\approx-g|E|\xbar{(h\vla{\theta})}_E\nabla Z_h(\mbf x_c)\\
 &\qquad=g\xbar{(h\vla{\theta})}_E\Big[\frac{2}{3}\big(Z_{\sigma_4}\mbf n_3+Z_{\sigma_5}\mbf n_1+Z_{\sigma_6}\mbf n_2\big)-\frac{1}{6}\big(Z_{\sigma_1}\mbf n_1+Z_{\sigma_2}\mbf n_2+Z_{\sigma_3}\mbf n_3\big)\Big].
\end{aligned}
\end{equation*}

\subsection{First-order scheme for point value}
For the evolution of the boundary values $\mbf v_\sigma$, the first-order spatial descretization is more involved. Following the approach introduced in \cite{Abgrall_Lin_Liu}, we temporarily re-number the boundary nodes for the definition of sub-elements. Starting with a selected DoF, denoted as $\varsigma_1$, we list the other nodes in a counter-clockwise manner, creating the list $\{\varsigma_1, \ldots, \varsigma_6\}$ of all boundary DoFs, with the centroid labeled as $\varsigma_7$. We then define six sub-elements, $\{T^E_i\}_{i=1}^6$, with vertices given by \bla{$T^E_1=\{\varsigma_1,\varsigma_2,\varsigma_7\}$, $T^E_2=\{\varsigma_2,\varsigma_3,\varsigma_7\}$, $T^E_3=\{\varsigma_3,\varsigma_4,\varsigma_7\}$, $T^E_4=\{\varsigma_4,\varsigma_5,\varsigma_7\}$, $T^E_5=\{\varsigma_5,\varsigma_6,\varsigma_7\}$, $T^E_6=\{\varsigma_6,\varsigma_1,\varsigma_7\}$}. For each sub-elements $T^E_i$, we can also get the inward normals $\mbf n_{\varsigma_i}$. For quadratic elements, and using back the original numbering, this is illustrated in Figure \ref{p2_ele_subcell}. 
\begin{figure}[ht!]
\centerline{\includegraphics[trim=0.01cm 0.01cm 0.01cm 0.01cm,clip,width=4.5cm]{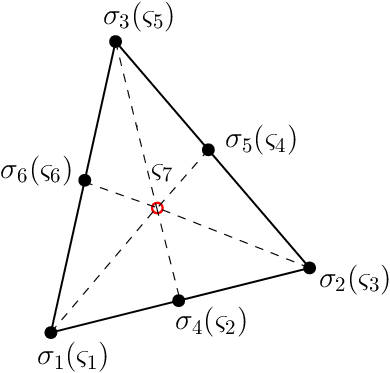}}
\caption{Geometry for the first-order scheme.\label{p2_ele_subcell}}
\end{figure}

\bla{
The boundary values $\mbf v_\sigma$ are updated using \eqref{2.12} in the first-order case, but with a modified $\Phi_\sigma^E$:
\begin{equation}\label{3.3}
\Phi_\sigma^E=\frac{1}{\vert C_\sigma\vert }\sum_{T^E_i, \sigma\in T^E_i}\Phi_{\sigma}^{T^E_i}(\mbf v_h)
\end{equation}
where $\vert C_\sigma\vert$ denotes the dual control volume associated with the DoF $\mbf v_\sigma$ and is defined as
\begin{equation}\label{3.7}
  \vert C_\sigma\vert =\sum_{E, \sigma\in E}\sum_{T^E_i, \sigma\in T^E_i} \frac{\vert T^E_i\vert}{3}=\sum_{E, \sigma\in E}\frac{\vert E\vert}{9}.
\end{equation} 

The residuals for each sub-triangle $T_i^E$ are defined as follows, inspired by monotone first-order residual distribution schemes, see \cite{Deconinck2007,Abgrall2006,RB09} for example. We use a version inspired by the local Lax-Friedrichs scheme:
\begin{equation}\label{3.4}
  \Phi_\sigma^{T^E_i}=\frac{1}{3}\int_{T^E_i} \big(\mbf J(\mbf v_h)\cdot\nabla\mbf v_h-\widetilde{\mbf S}(\mbf{v}_h, Z_h)\big) {\rm d}\mbf x+\alpha_{T^E_i}\big (\mbf {v}_\sigma-\xbar{\mbf v}_{T^E_i}\big )
\end{equation}
and apply a $P^1$ approximation on $T^E_i$, 
where $\xbar{\mbf v}_{T^E_i}$ (resp. $\xbar{Z}_{T^E_i}$) is the arithmetic average of the $\mbf v$'s (resp. $Z$'s) at three vertices of $T^E_i$ (hence we use the average value here).}
Finally, 
\begin{equation}\label{3.6}
  \alpha_{T^E_i}=\max\limits_{\substack{\sigma\in T^E_i,\\\mbf n_{b_i} \text{normals of~} T^E_i}}\rho\big ( \mbf J(\mbf v_\sigma)\cdot \mbf n_{b_i}\big ).
\end{equation}

It is well-known that the local Lax-Friedrichs' finite volume scheme, defined by \eqref{2.7} and \eqref{3.1}--\eqref{3.2}, for evolving the cell average of water depth $h$ and bulk temperature $h\theta$, is positivity-preserving under the standard Courant--Friedrichs--Lewy (CFL) condition. \vla{The positivity-preserving property of local Lax--Friedrichs' schemes for the evolution of point values is discussed in \cite[Section 4.2]{RB09} and can be omitted here for brevity.}

\section{Numerical examples}\label{sec4}
In this section, we demonstrate the performance of the proposed third-order WB PAMPA scheme through several numerical examples. In all test cases, cell averages are updated using the conservative formulation \eqref{1.1}--\eqref{1.1a}, while point values are evolved according to the non-conservative pressure-momentum-temperature formulation \eqref{2.1}--\eqref{2.2a}. Numerical results obtained with the WB PAMPA scheme and  its non-WB counterparts will be compared.

In all numerical experiments, unless otherwise specified, the zero-order extrapolation boundary conditions are used. All the meshes are unstructured and generated by GMSH \cite{Geuzaine2009}. \vla{In Examples 1--5, we set the acceleration due to gravity as $g=9.812$ while in Examples 6--9, we take $g=1$.} \bla{For some results, we compute the discrete $L^1$- and $L^{\infty}$-errors for the cell averages $\xbar{\mbf u}_E$ and $\xbar{\mbf u}'_E$ by
\begin{equation*}
  \Vert\xbar{\mbf u}_E-\xbar{\mbf u}'_E\Vert_{L^1}=\frac{\sum\limits_{E\in\mathcal{E}}\vert E\vert\vert\xbar{\mbf u}_E-\xbar{\mbf u}'_E\vert}{\sum\limits_{E\in \mathcal{E}}\vert E\vert},\quad 
  \Vert\xbar{\mbf u}_E-\xbar{\mbf u}'_E\Vert_{L^\infty}=\max\limits_{E\in\mathcal{E}}\vert\xbar{\mbf u}_E-\xbar{\mbf u}'_E\vert
\end{equation*}
and for the point values ${\mbf u}_\sigma$ and ${\mbf u}'_\sigma$
\begin{equation*}
  \Vert{\mbf u}_\sigma-{\mbf u}'_\sigma\Vert_{L^1}=\frac{\sum\limits_{E\in\mathcal{E},\sigma\in E}\vert C_\sigma\vert\vert{\mbf u}_\sigma-{\mbf u}'_\sigma\vert}{\sum\limits_{E\in \mathcal{E},\sigma\in E}\vert C_\sigma\vert},\quad \Vert{\mbf u}_\sigma-{\mbf u}'_\sigma\Vert_{L^\infty}=\max\limits_{E\in\mathcal{E},\sigma\in E}\vert{\mbf u}_\sigma-{\mbf u}'_\sigma\vert.
\end{equation*}
respectively.}

\subsubsection*{Example 1---Accuracy Test}
The goal of the first example is to experimentally check the high-order accuracy of the proposed schemes when applied to the following two-dimensional problem, \bla{taken from \cite{Clain_MOOD_SWE}}:
\begin{equation*}
 Z(r)=0.2e^{\frac{1-r^2}{2}},~ h(r,0)=1-\frac{1}{4g}e^{2(1-r^2)}-Z(r),~ \bm\nu(r,0)=-\mbf x^\bot e^{1-r^2},~ \theta(r,0)=1,
\end{equation*}
where $r=\sqrt{x^2+y^2}$ and $\mbf x^\bot=(-y,x)$. The computational domain is $[-10,10]\times[-10,10]$. Note that this is a stationary equilibrium of the system \eqref{1.1}--\eqref{1.1a} so that the exact solution coincides with the initial condition at any time. \bla{Dirichlet boundary conditions are imposed by evaluating the exact solution at the quadrature points on the domain boundary.}

We compute the numerical solution until the final time $t =1$ using the proposed WB PAMPA scheme. The unstructured meshes are obtained using the Frontal-Delaunay \bla{option of GMSH} with mesh characteristic lengthes $1.0$, $0.5$, $0.25$, and $0.15$. We measure the discrete $L^1$-errors for cell averages as well as the point values of the conservative variables and then compute the rate of convergence. Tables \ref{tab1} and \ref{tab2} report the results obtained by the WB PAMPA scheme, from which we can see that the expected experimental third order of accuracy is achieved for the studied WB PAMPA scheme.

\begin{table}[ht!]
\color{red}
\caption{\label{tab1} Example 1: $L^1$ errors and numerical orders of accuracy for cell averages.}
\begin{center}
\begin{tabular}{|c||c|c|c|c|c|c|}
\hline
$\max\limits_j\sqrt{\vert E_j\vert}$& $h$ &rate&  $hu$& rate&$hv$& rate \\
\hline
$0.7510$     &  $1.693\times10^{-4}$&- &  $6.368\times10^{-5}$& -     &        $7.482\times10^{-4}$  & - \\
$0.3681 $ &  $3.352\times10^{-5}$&$2.27$&  $1.310\times10^{-4}$& $2.22$&       $1.265 \times10^{-4}$  & $2.49$\\
$0.1954$ &  $5.296\times10^{-6}$& $2.91$&  $1.831\times10^{-5}$& $3.11$&       $1.826\times10^{-5}$  & $3.05$\\
$0.1250$ &  $1.484\times10^{-6}$& $2.85$&  $5.152\times10^{-6}$& $2.84$&       $5.144\times10^{-6}$  & $2.84$\\
\hline
\end{tabular}
\end{center}
\end{table}

\begin{table}[ht!]
\color{red}
\caption{\label{tab2} Example 1: Same as in Table \ref{tab1} but for point values.}
\begin{center}
\begin{tabular}{|c||c|c|c|c|c|c|}
\hline
$\max\limits_\sigma\sqrt{\vert C_\sigma\vert}$& $h$ &rate&  $hu$& rate&$hv$& rate \\
\hline
$0.6128$     &  $1.720\times10^{-4}$&- &  $1.040\times10^{-3}$& -     &        $1.093\times10^{-3}$  & - \\
$0.3025 $ &  $3.376\times10^{-5}$&$2.31$&  $1.714\times10^{-4}$& $2.55$&       $1.755 \times10^{-4}$  & $2.59$\\
$0.1552$ &  $5.275\times10^{-6}$& $2.78$&  $2.442\times10^{-5}$& $2.92$&       $2.456\times10^{-5}$  & $2.95$\\
$0.9903$ &  $1.484\times0^{-6}$& $2.82$&  $6.853\times10^{-6}$& $2.83$&       $6.907\times10^{-6}$  & $2.82$\\
\hline
\end{tabular}
\end{center}
\end{table}

\subsubsection*{Example 2---Traveling Vortex}
In the second example, we consider a traveling vortex initially located at $(0,0)$ that propagates with a speed of $\bm\nu_{\rm ad}=(1,\frac{\sqrt{2}}{2})$ in a square computational domain $[-10,10]\times[-10,10]$. The initial conditions are a slightly modification of those from the previous example, given by 
\begin{equation}\label{Ex2ini}
h(r,0)=1-\frac{1}{4g}e^{2(1-r^2)},~ \bm\nu(r,0)=\bm\nu_{\rm ad}-\mbf x^\bot e^{1-r^2},~ \theta(r,0)=1,
\end{equation}
\bla{and the same Dirichlet boundary conditions as in the previous example are used.}

We first consider a flat bottom topography, i.e., $Z(x,y)=0$. In this case, the vortex travels in the direction $(1,\frac{\sqrt{2}}{2})$, and we can derive the exact solution, which is defined as
\begin{equation*}
  h(\mbf x,t)=h(\mbf x-t\bm\nu_{\rm ad},0),
  \quad \bm\nu(\mbf x,t)=\bm\nu(\mbf x-t\bm\nu_{\rm ad},0),
  \quad\theta(\mbf x,t)=\theta(\mbf x-t\bm\nu_{\rm ad},0).
\end{equation*}

In Figure \ref{Ex2_h}, we show the computed water depth $h$ and the exact solution (represented by the black curve, which coincides with the numerical solution) on a mesh with 23246 triangular elements and 46893 DoFs at the final time $t=5$. As one can observe, the studied WB PAMPA scheme can capture the traveling vortex very accurately. The vortex moves in the correct direction and reaches the right final position. Moreover, at this scale, the difference between the exact and numerical solutions is \bla{insignificant}.  

\begin{figure}[ht!]
\centerline{\subfigure[$\xbar{h}_E$]{\includegraphics[trim=1.8cm 2.5cm 2.5cm 2.5cm,clip,width=5.0cm]{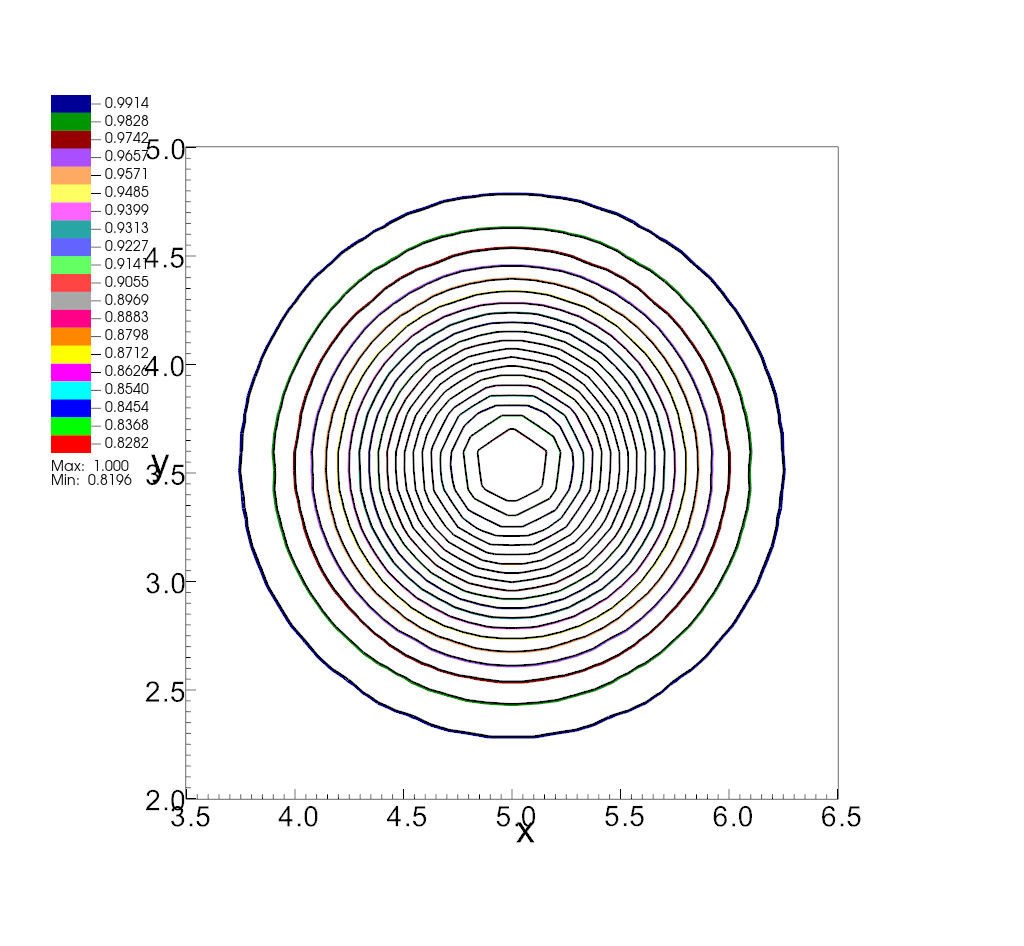}}\hspace*{0.5cm}
	\subfigure[$h_\sigma$]{\includegraphics[trim=1.8cm 2.5cm 2.5cm 2.5cm,clip,width=5.0cm]{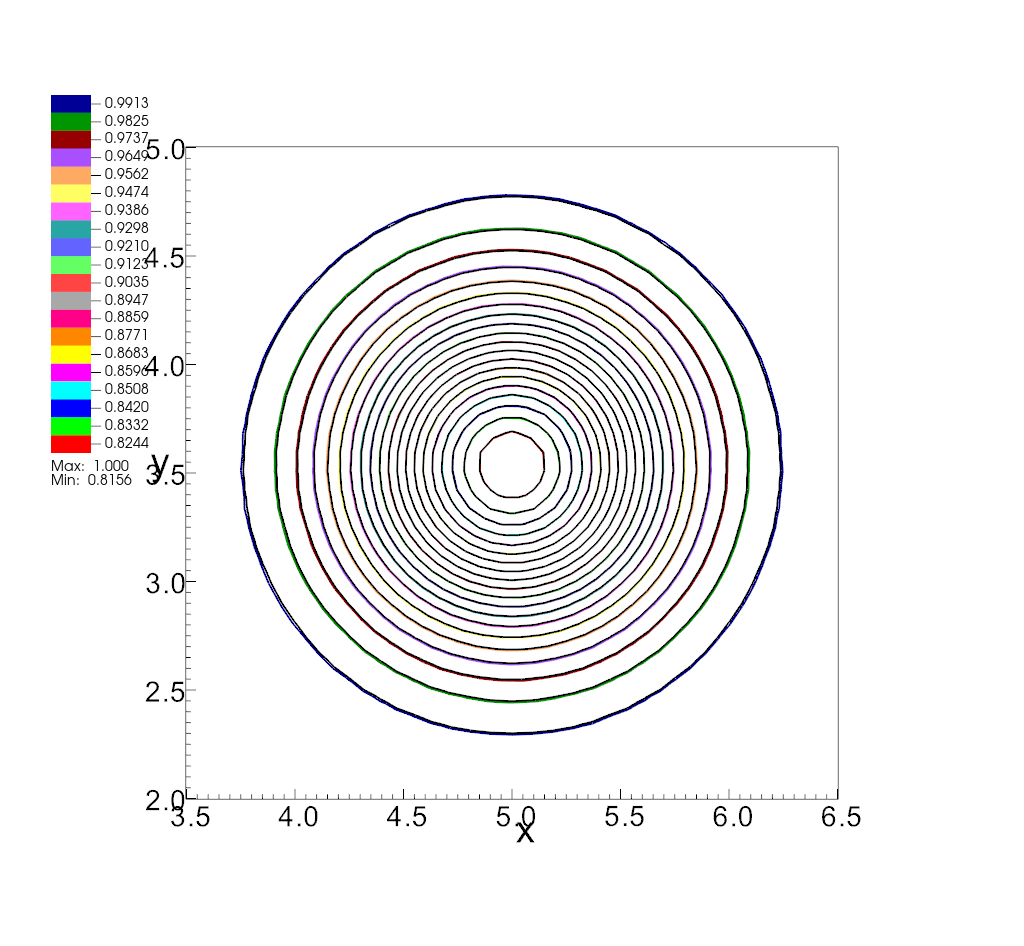}}}
\caption{\sf Example 2: Cell averages and point values of water depth at $t=5$.\label{Ex2_h}}
\end{figure}

Next, we introduce a non-flat bottom topography, defined as
\begin{equation*}
  Z(\mbf x)=0.2e^{\frac{1}{2}-2\Vert\mbf x-2\Vert^2},
\end{equation*}
while keeping \bla{the rest as} in \eqref{Ex2ini}. In this case, the water depth is perturbed due to the non-flat bottom, but the vortex still travels in the same direction. In Figure \ref{Ex2_h1}, we show the water depth at times $t=0.39$ and $t=1$. We can see from the top row of Figure \ref{Ex2_h1} that at $t=0.39$, the front part of the vortex has just passed through the bottom topography and the tail part remains behind it. The bottom row of Figure \ref{Ex2_h1} presents the water depth at time $t=1$, corresponding to the moment when the tail part of the vortex is about to pass over the bottom. From these results, we can conclude that the studied WB PAMPA scheme can accurately capture the dynamical traveling processes of the vortex over a non-flat bottom topography. This demonstrates great potential for applying these schemes to the study of rotating shallow water equations and the simulation of cyclonic vortices; see, e.g., \cite{KLZ_hurricane, KLZ_2D}.    
\begin{figure}[ht!]
\centerline{\subfigure[$\xbar{h}_E$ at $t=0.39$]{\includegraphics[trim=1.8cm 2.5cm 2.5cm 2.5cm,clip,width=5.0cm]{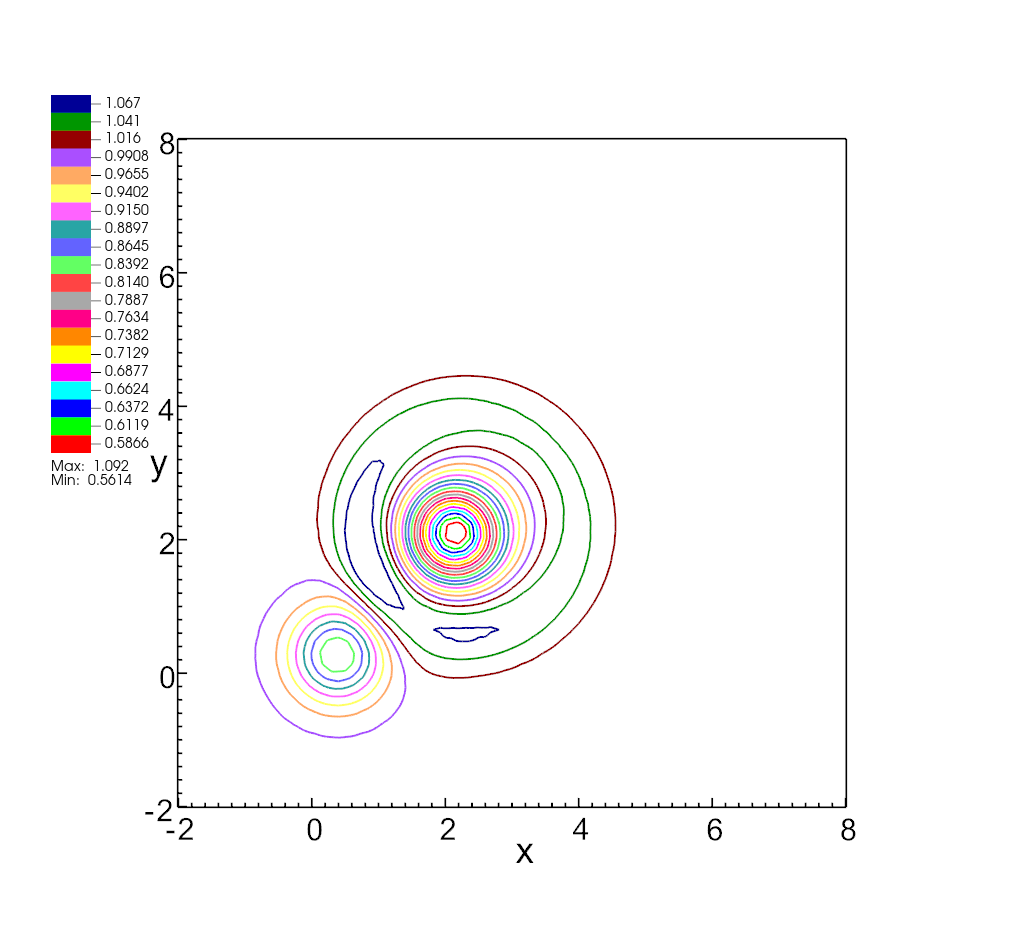}}\hspace*{0.5cm}
	\subfigure[$h_\sigma$ at $t=0.39$]{\includegraphics[trim=1.8cm 2.5cm 2.5cm 2.5cm,clip,width=5.0cm]{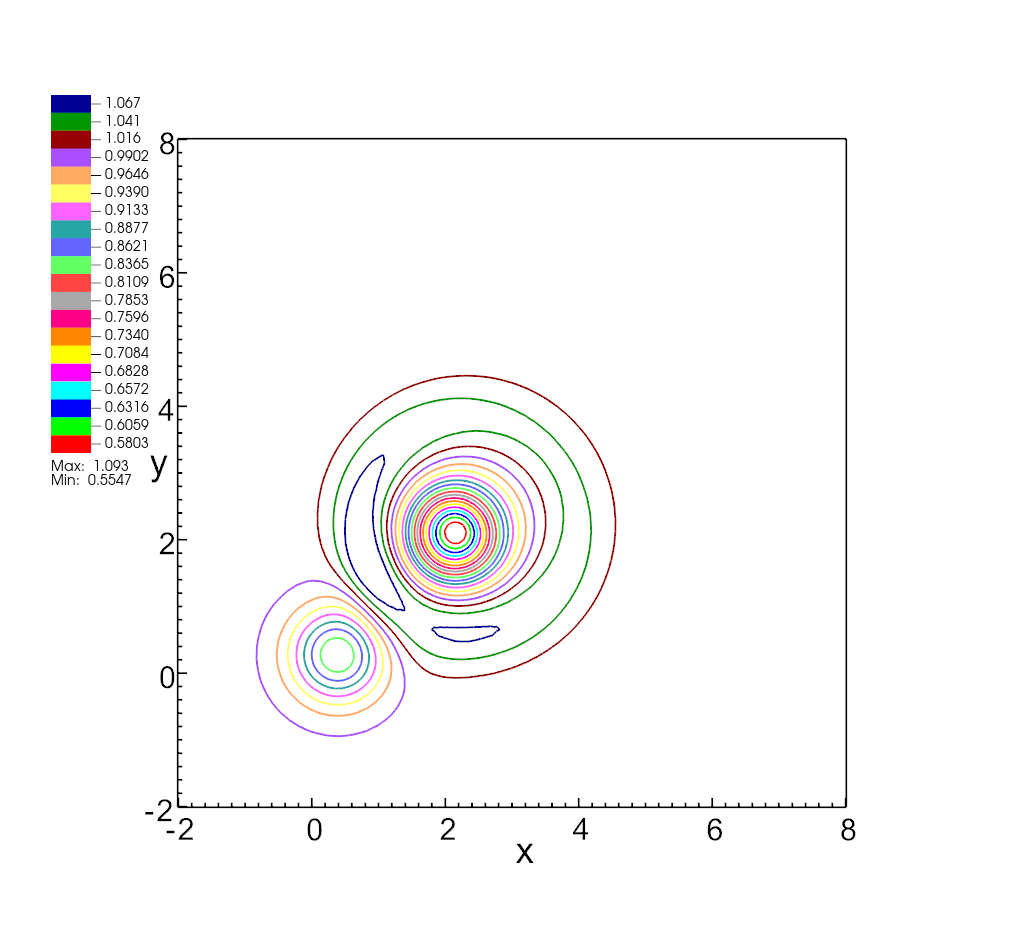}}}
\vskip1pt
\centerline{\subfigure[$\xbar{h}_E$ at $t=1$]{\includegraphics[trim=1.8cm 2.5cm 2.5cm 2.5cm,clip,width=5.0cm]{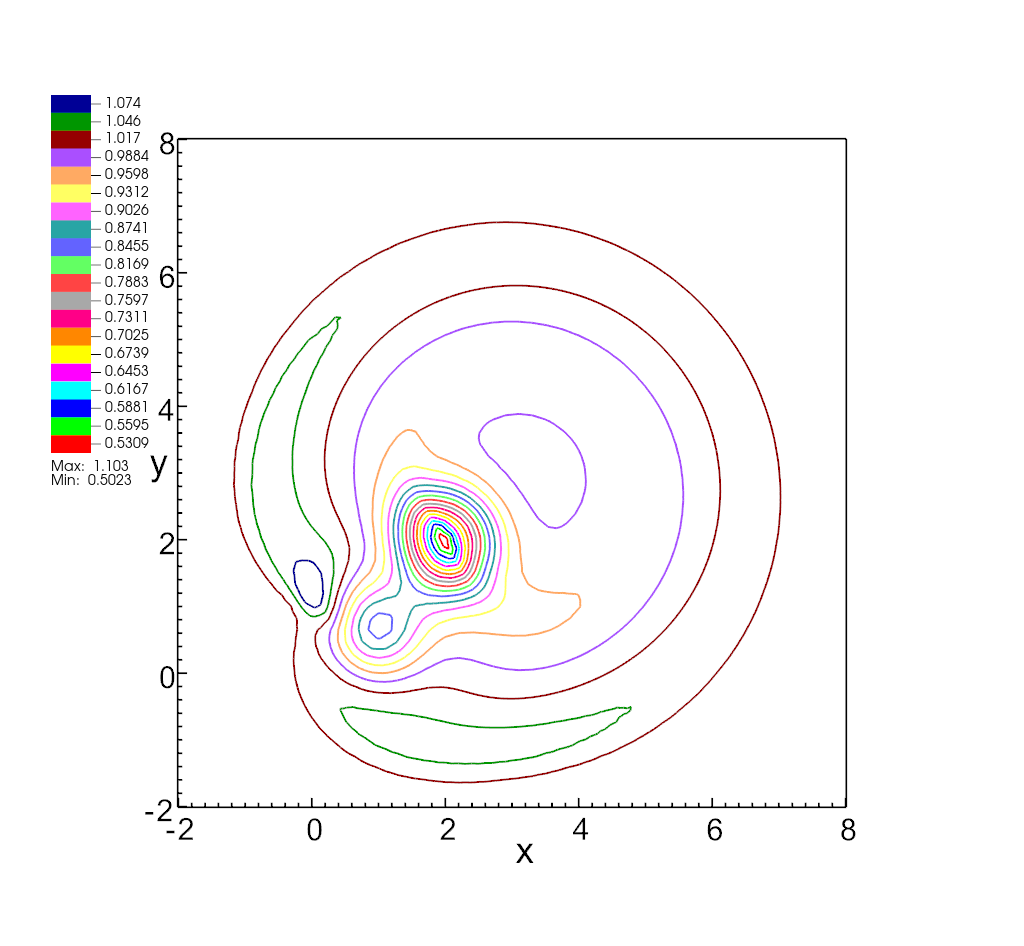}}\hspace*{0.5cm}
	\subfigure[$h_\sigma$ at $t=1$]{\includegraphics[trim=1.8cm 2.5cm 2.5cm 2.5cm,clip,width=5.0cm]{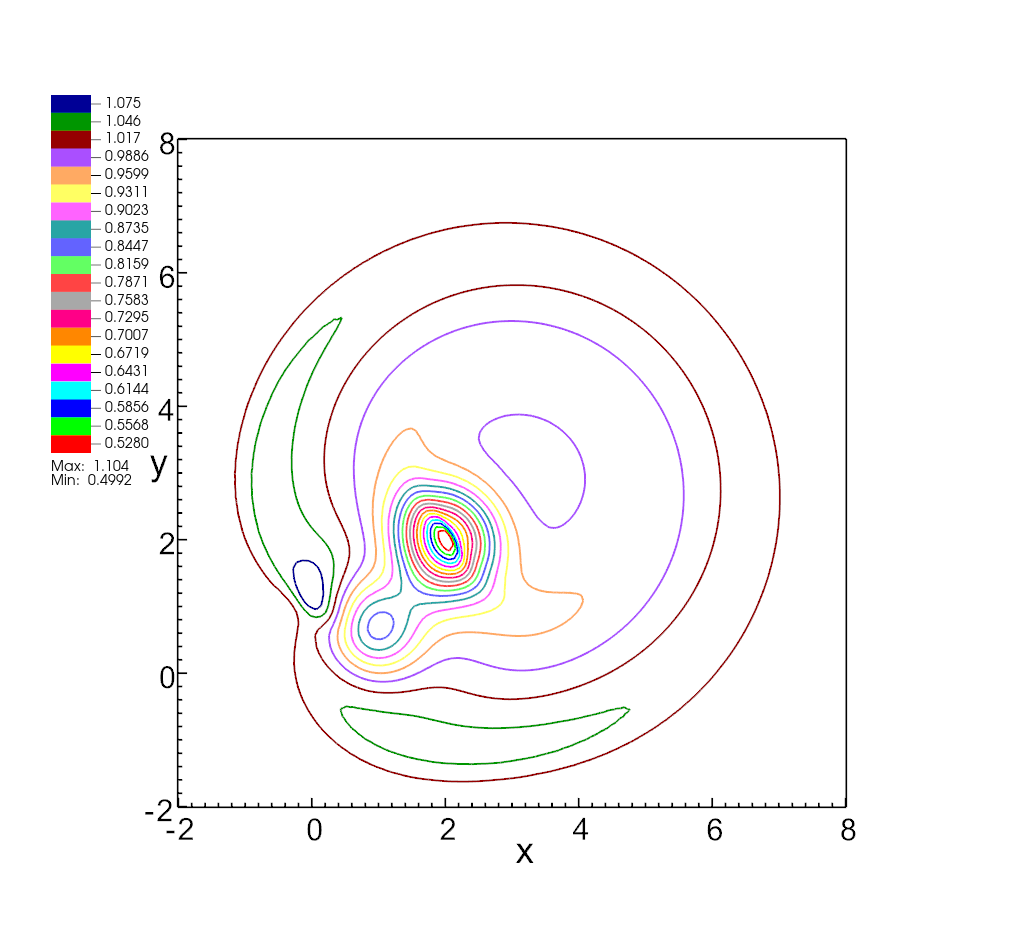}}}
\caption{\sf Example 2: Cell averages and point values of water depth at $t=0.39$ (top row) and $t=1$ (bottom row) computed by the WB PAMPA scheme over non-flat bottom topography.\label{Ex2_h1}}
\end{figure}


\subsubsection*{Example 3---Well-balanced Test: Flows over Three Humps}
In the third example, \bla{taken from \cite{DelGrosso2024}}, we verify the WB property of the proposed two schemes towards the ``lake at rest'' steady-state solution. The computational domain is $[0,40]\times[0,40]$ and the bottom topography is defined as:
\begin{equation*}
  Z(\mbf x)=\max\Big(0,1-\frac{1}{8}r_{10,11},1-\frac{3}{10}r_{10,31},1-\frac{4}{10}r_{27,20}\Big),
\end{equation*} 
where $r_{a,b}=\sqrt{(x-a)^2+(y-b)^2}$. The initial data satisfy a lake at rest solution, that is,
\begin{equation*}
  h(\mbf x,0)=4-Z(\mbf x),\quad \bm\nu(\mbf x,0)=\mbf 0,\quad \theta(\mbf x,0)=1. 
\end{equation*} 

We compute the solution until a final time $t=20$ on the triangular mesh consisting of $3718$ elements and $7597$ DoFs. In Figure \ref{Ex3_h}, we present the results of water depth with 20 isolines obtained by the WB PAMPA scheme. We can see that the initial profile of water depth is preserved. 

\begin{figure}[ht!]
\centerline{\subfigure[$\xbar{h}_E$ at $t=20$]{\includegraphics[trim=1.8cm 2.5cm 2.5cm 2.5cm,clip,width=5.0cm]{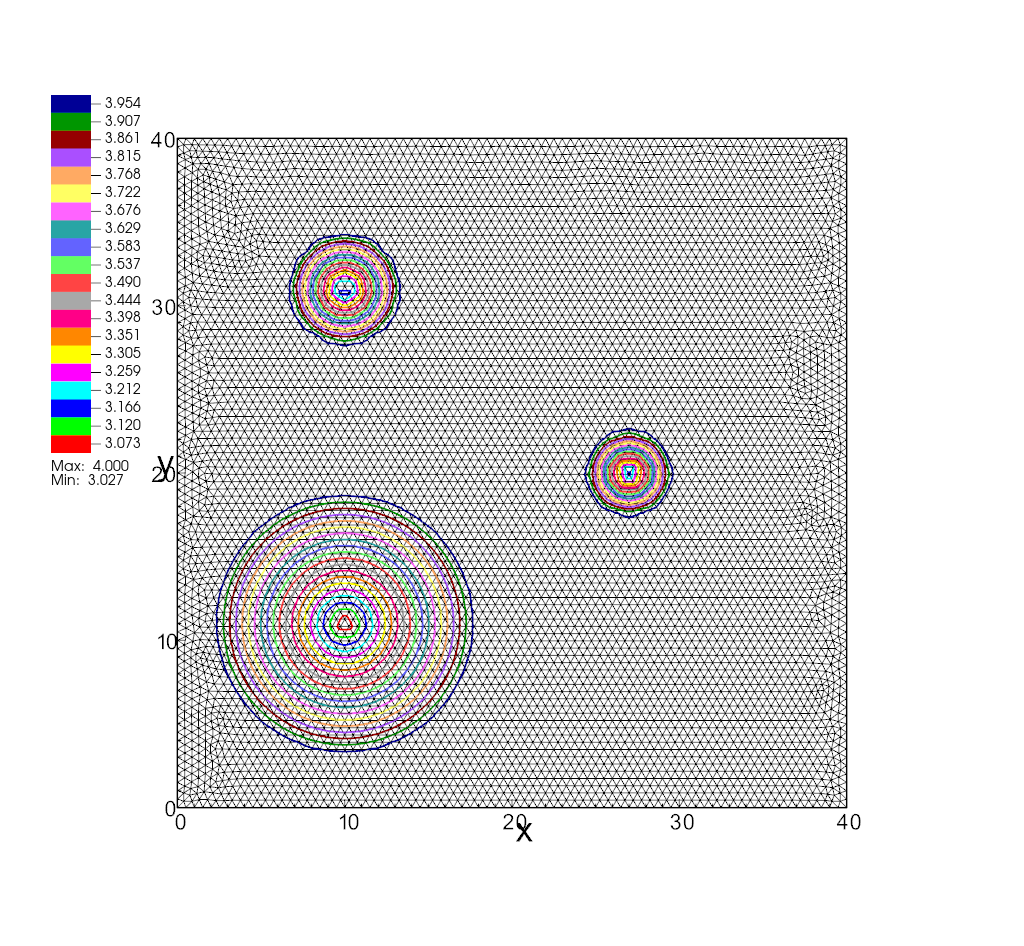}}\hspace*{0.5cm}
	\subfigure[$h_\sigma$ at $t=20$]{\includegraphics[trim=1.8cm 2.5cm 2.5cm 2.5cm,clip,width=5.0cm]{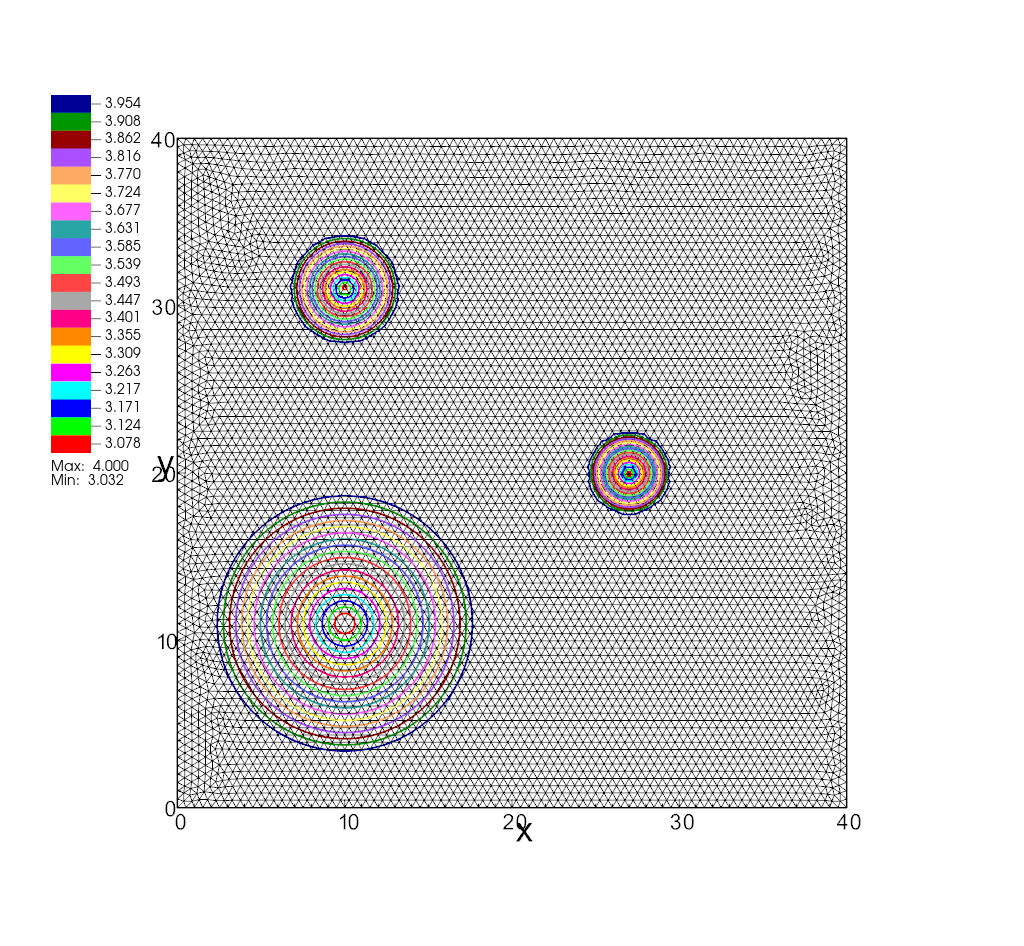}}}
\caption{\sf Example 3: Numerical solution (cell averages and point values of water depth) at $t=20$ computed by the WB PAMPA scheme.\label{Ex3_h}}
\end{figure}

Furthermore, in order to demonstrate that the lake at rest equilibria is indeed maintained up to round-off error, we compute the discrete $L^1$- and $L^\infty$-errors for the numerical solution (cell averages and point values of water depth $h$ and momentum $hu$ (similar behavior for $hv$ and omitted here)). \bla{For comparison, a non-WB version of the proposed PAMPA scheme is also tested. In this version, we use only the 3-point Gauss--Lobatto quadrature rule to compute the edge integral in \eqref{2.7}. While this quadrature rule maintains numerical accuracy, it fails to exactly evaluate the edge integral at the ``lake at rest'' state, rendering the scheme non-WB.} We show the obtained results in Tables \ref{Ex3:tab1} and \ref{Ex3:tab2}. As one can clearly see, all errors computed by the WB PAMPA scheme are within machine accuracy, which verifies its WB property. \bla{The non-WB scheme fails to maintain the errors within machine accuracy, further demonstrating the necessity of exactly evaluating the edge integral and employing \cref{alg:Gaussian_integrals}.}

\begin{table}[ht!]
\caption{\sf Example 3: $L^1$- and $L^\infty$-errors in $\mbf u_E$ for the ``lake at rest'' state obtained by WB PAMPA scheme and its non-WB counterpart.\label{Ex3:tab1}}
\begin{center}
\begin{tabular}{|c| c c| c c| }\hline
\multicolumn{1}{|c|}{\multirow{2}{*}{Schemes}}  &\multicolumn{2}{c|}{$h$} &\multicolumn{2}{c|}{$hu$} \\ \cline{2-5}
  & $L^1$-error   & $L^\infty$-error  & $L^1$-error  &$L^\infty$-error  \\ \hline
 \multicolumn{1}{|c|}{\multirow{1}{*}{WB}}  &$1.22\times10^{-15}$ & $8.84\times10^{-14}$ & $7.96\times10^{-14}$ & $9.16\times10^{-12}$  \\  \hline
 \multicolumn{1}{|c|}{\multirow{1}{*}{non-WB}}  &$1.86\times10^{-5}$ & $3.84\times10^{-4}$ & $6.35\times10^{-5}$& $2.56\times10^{-3}$\\ \hline
\end{tabular}
\end{center}
\end{table}

\begin{table}[ht!]
\caption{\sf Example 3: Same as in Table \ref{Ex3:tab1} but for the point values $\mbf u_\sigma$.\label{Ex3:tab2}}
\begin{center}
\begin{tabular}{|c| c c| c c|}\hline
\multicolumn{1}{|c|}{\multirow{2}{*}{Schemes}}  &\multicolumn{2}{c|}{$h$} &\multicolumn{2}{c|}{$hu$} \\ \cline{2-5}
  & $L^1$-error   & $L^\infty$-error  & $L^1$-error  &$L^\infty$-error  \\ \hline
 \multicolumn{1}{|c|}{\multirow{1}{*}{WB}}  &$1.20\times10^{-16}$ & $3.55\times10^{-15}$ & $7.37\times10^{-14}$& $5.61\times10^{-13}$ \\  \hline
 \multicolumn{1}{|c|}{\multirow{1}{*}{non-WB}}  &$1.13\times10^{-6}$ & $9.08\times10^{-4}$ & $4.25\times10^{-4}$& $1.79\times10^{-2}$ \\  \hline
\end{tabular}
\end{center}
\end{table}

\subsection*{Example 4---Small Perturbation of a Steady-State Solution}
In the fourth example, we perform an additional test to verify the WB property of the proposed schemes by imposing a perturbation into the stationary solution. Non well-balanced schemes often struggle to capture the correct propagation of a perturbation when its magnitude is smaller than the spatial discretization size. While mesh refinement can improve this issue, it is typically not a practically affordable solution. Therefore, well-balanced schemes are highly valuable. In what follows, we consider a non-flat bottom topography defined by
\begin{equation*}
  Z(\mbf x)=0.8e^{-50\Vert\mbf x-1\Vert^2},
\end{equation*}  
prescribed in a computational domain $[0,2]\times[0,2]$. The initial conditions are given by
\begin{equation*}
  h(\mbf x,0)=1-Z(\mbf x)+10^{-4}e^{-20\Vert\mbf x-0.8\Vert^2},\quad \bm\nu(\mbf x,0)=\mbf0,\quad \theta(\mbf x,0)=1.
\end{equation*}

We compute the numerical results using the WB PAMPA scheme and a non well-balanced scheme \bla{as described in the previous example. We} show the propagation of the perturbation in the cell averages of water depth $h$ at time $t=0.1$. Figure \ref{Ex4_h} shows the results obtained using a coarse mesh with 944 triangular elements and 1969 DoFs.  We can clearly see that the WB PAMPA scheme can accurately capture the propagation of small perturbations, while the non well-balanced scheme shows substantial differences. 

\begin{figure}[ht!]
\centerline{\subfigure[WB PAMPA]{\includegraphics[trim=0.2cm 1.5cm 2.5cm 2.5cm,clip,width=5.0cm]{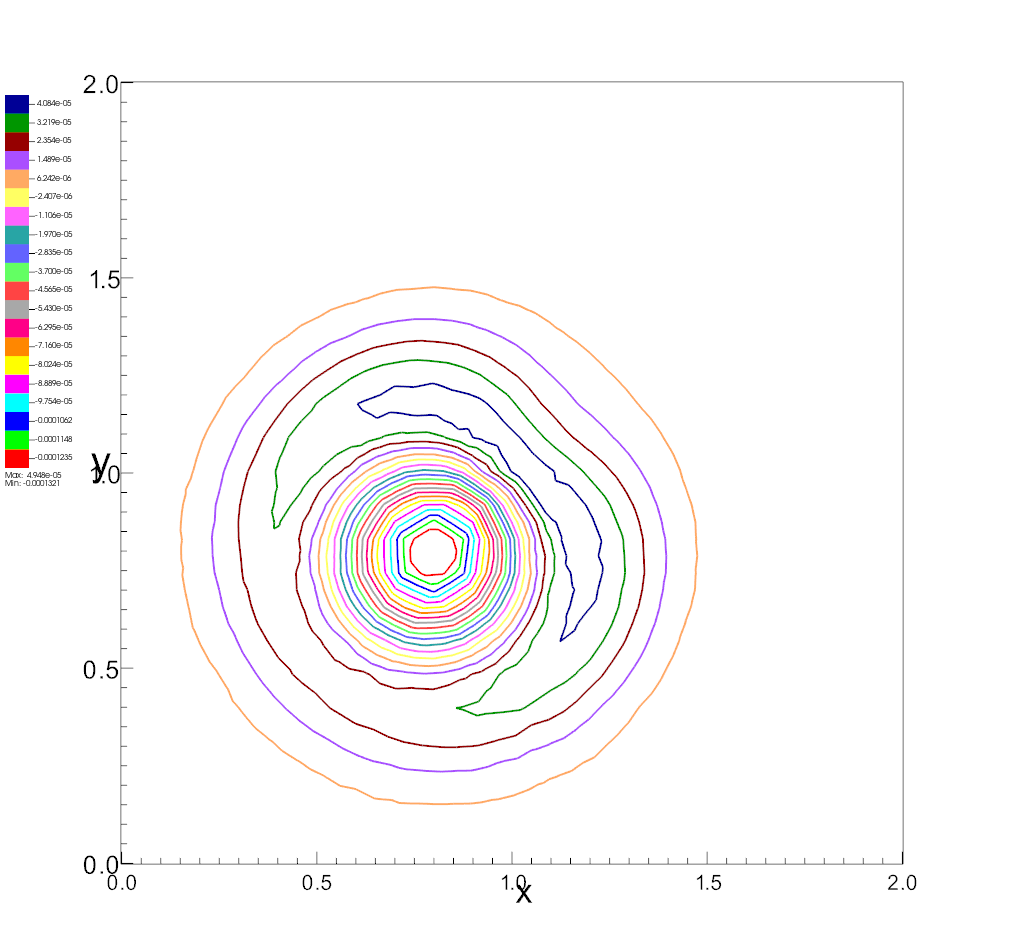}}\hspace*{0.5cm}
	\subfigure[non WB]{\includegraphics[trim=0.2cm 1.5cm 2.5cm 2.5cm,clip,width=5.0cm]{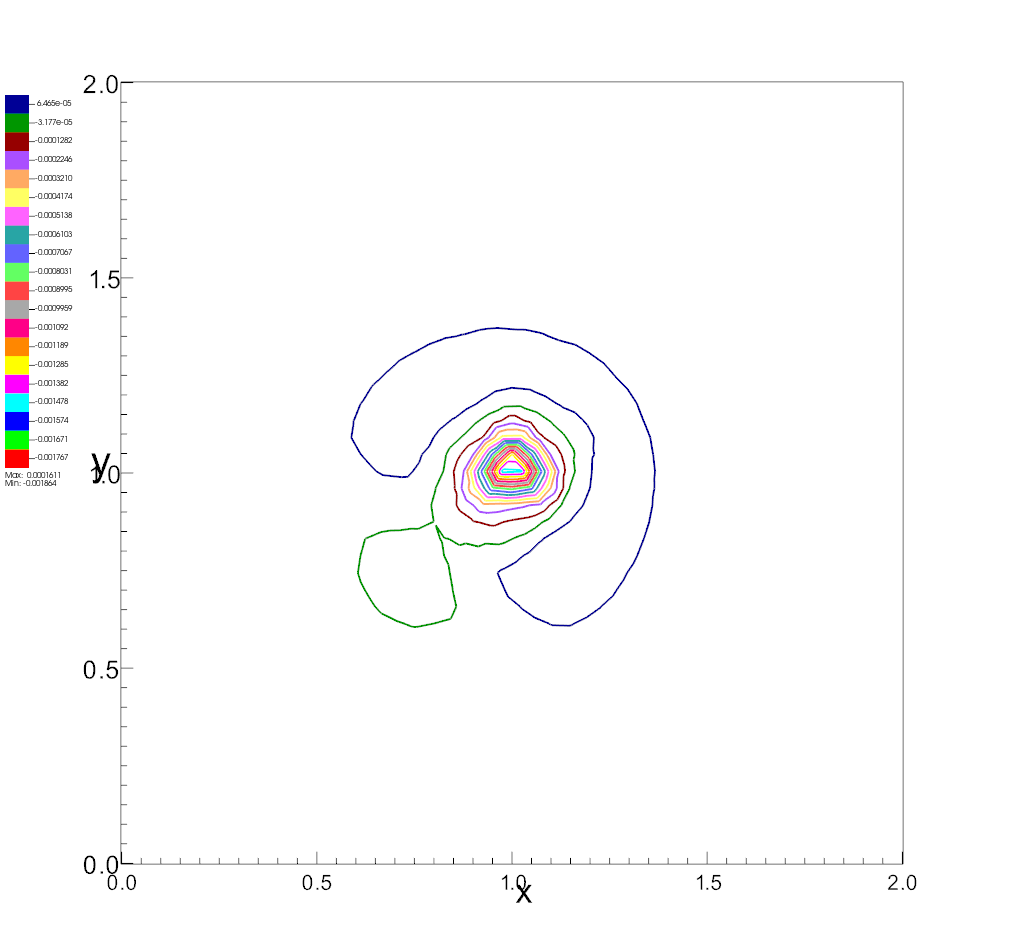}}}
\caption{\sf Example 4: Perturbation in the cell averages of water depth at $t=0.1$ computed by the WB PAMPA (left) and non well-balanced (right) schemes. A coarse mesh with 944 triangular elements and 1969 DoFs is used.\label{Ex4_h}}
\end{figure}

To further illustrate the effectiveness of WB schemes, we refine the mesh to 5824 elements and 11849 DoFs and plot the results in Figure \ref{Ex4_h2}. As one can observe, the result computed by the non well-balanced scheme now align more closely with those obtained using the WB PAMPA scheme. For the WB PAMPA scheme, the numerical solution remains consistent between the coarse and fine meshes. This also demonstrate the ability of WB schemes to accurately capture small perturbations with lower computational cost.

\begin{figure}[ht!]
\centerline{\subfigure[WB PAMPA]{\includegraphics[trim=0.2cm 1.5cm 2.5cm 2.5cm,clip,width=5.0cm]{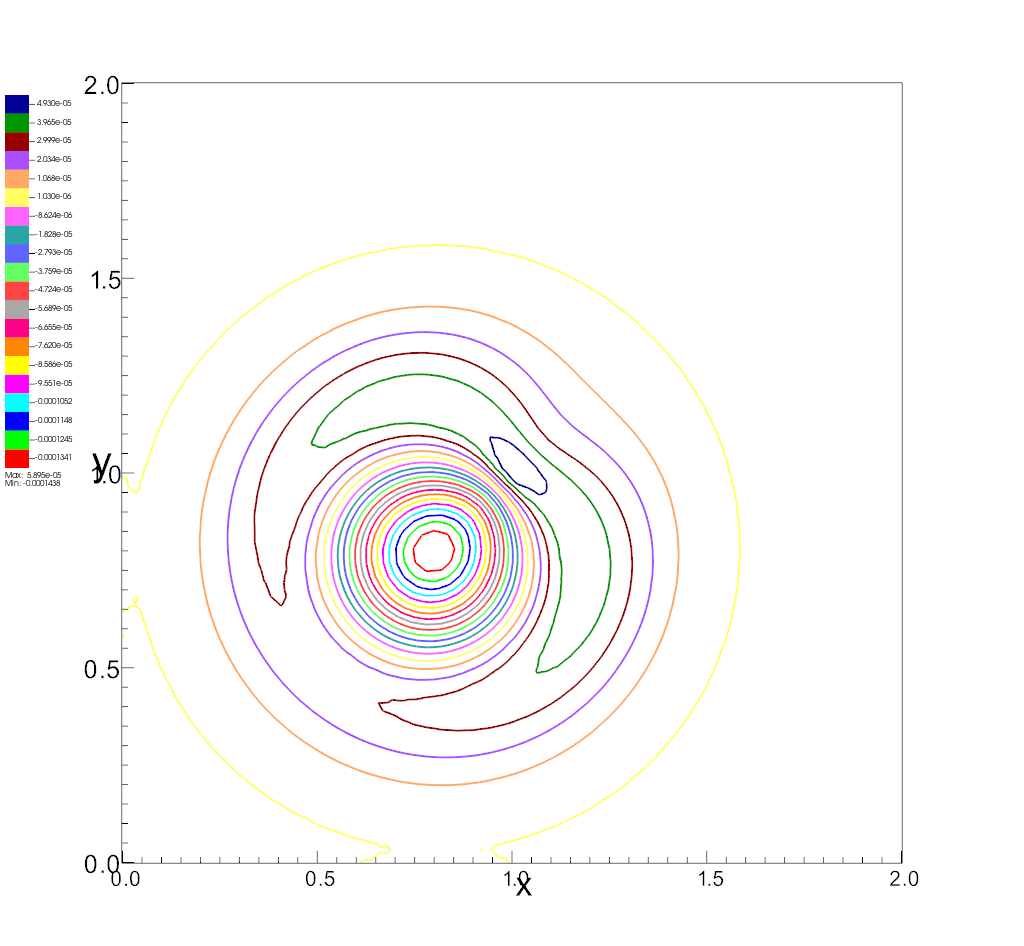}}\hspace*{0.5cm}
	\subfigure[non WB]{\includegraphics[trim=0.2cm 1.5cm 2.5cm 2.5cm,clip,width=5.0cm]{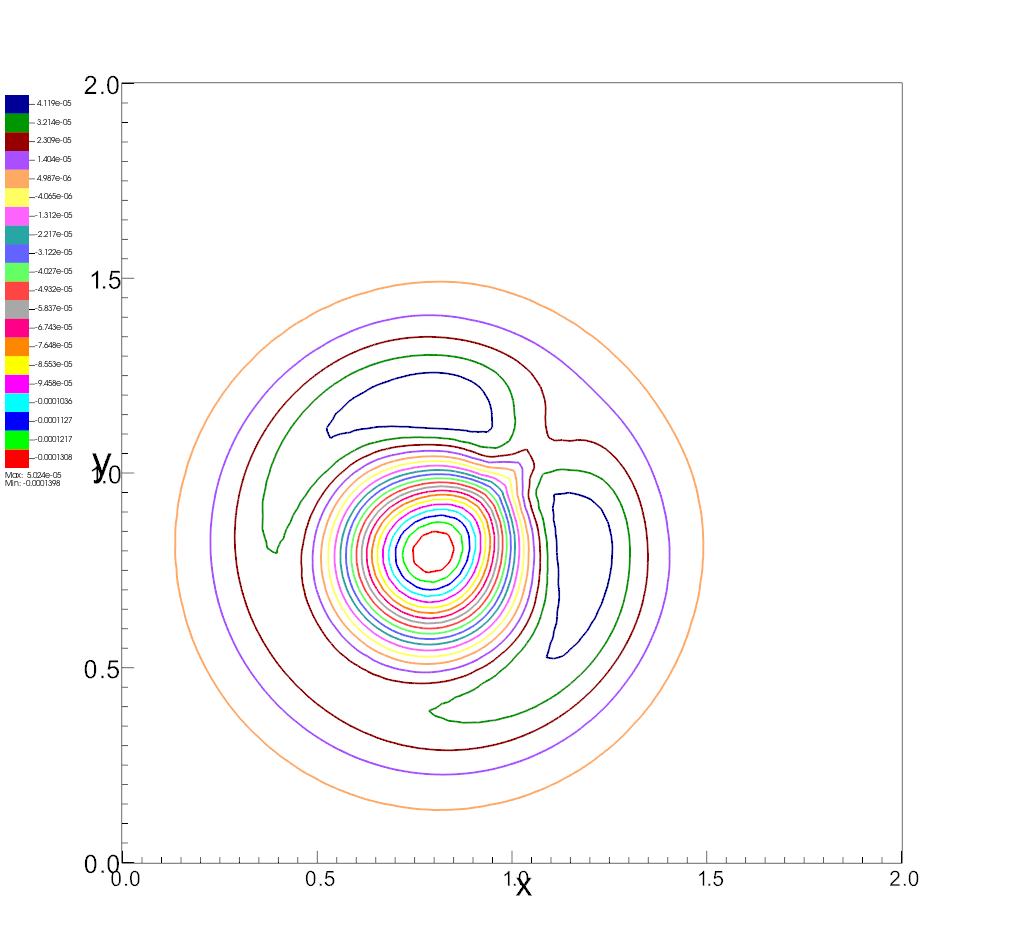}}}
\caption{\sf Example 4: Same as in Figure \ref{Ex4_h} but using a fine mesh with 5824 elements and 11849 DoFs.\label{Ex4_h2}}
\end{figure}

\subsection*{Example 5---Circular Dam-break Problem}
In the fifth example, we consider a circular dam-break problem on a flat bottom topography (i.e., $Z(\mbf x)=0$), which serves as a benchmark test case to assess the behaviour of numerical schemes. The computational domain is a square $[0,50]\times[0,50]$, divided into two regions by a cylindrical wall of radius $r_0=11$. The initial conditions are given by
\begin{equation*}
  h(\mbf x,0)=\left\{
  \begin{aligned}
  &10 &&\mbox{if}~\Vert\mbf x-25\Vert\leq r_0,\\
  &1  &&\mbox{otherwise},
  \end{aligned}\right. \quad \bm\nu(\mbf x,0)=\mbf 0,\quad \theta(\mbf x,0)=1.
\end{equation*}  

We run simulations using the WB PAMPA scheme on a triangular mesh with 23266 elements and 46933 DoFs. As time progresses from the initial state $t=0$, the water begins to drain, and the waves start to propagate radially and symmetrically. We plot the numerical solution of water depth at \vla{times $t=0.447$ and $t=0.69$ in the top row of Figures \ref{Ex5_h} and \ref{Ex2_h1}}. The studied scheme produce correct results, with the waves propagating uniformly and symmetrically. 

\begin{figure}[ht!]
\centerline{\subfigure[$\xbar{h}_E$ at $t=0.447$]{\includegraphics[trim=1.8cm 2.25cm 1.1cm 1.2cm,clip,width=5.0cm]{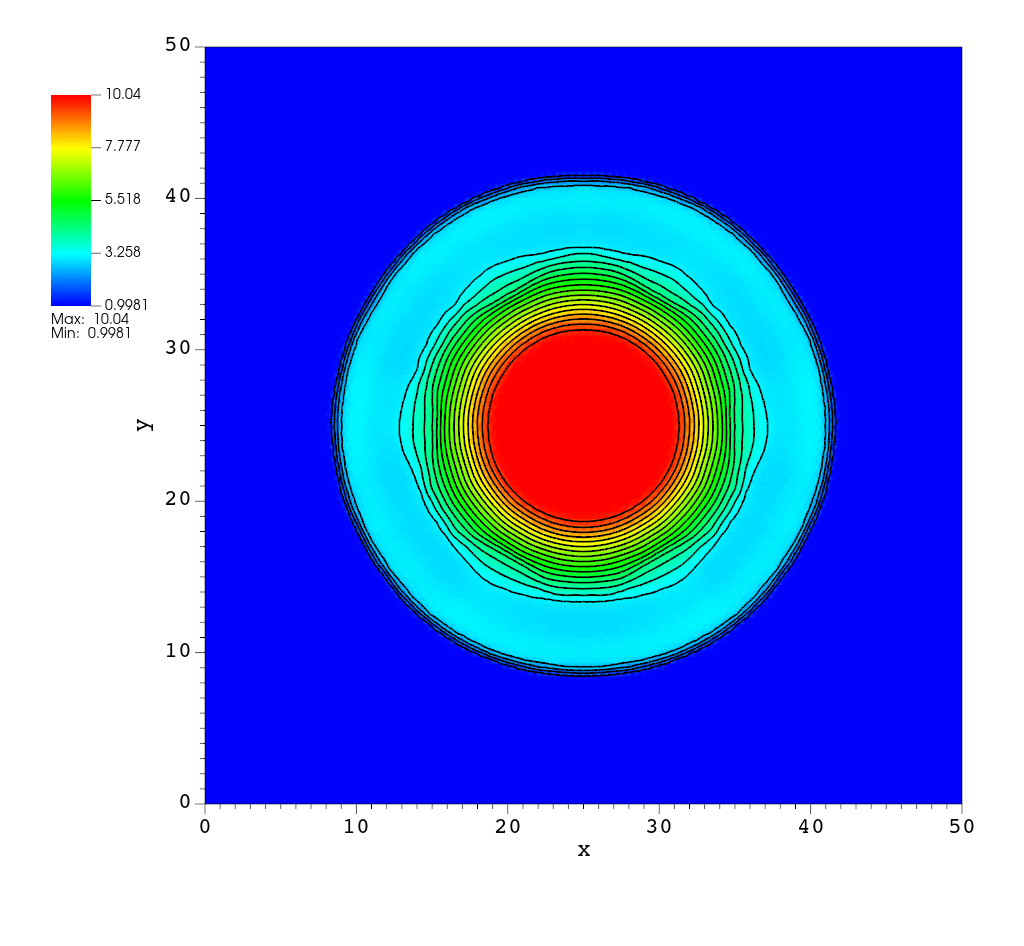}}\hspace*{0.15cm}
	\subfigure[$h_\sigma$ at $t=0.447$]{\includegraphics[trim=1.8cm 2.25cm 1.1cm 1.2cm,clip,width=5.0cm]{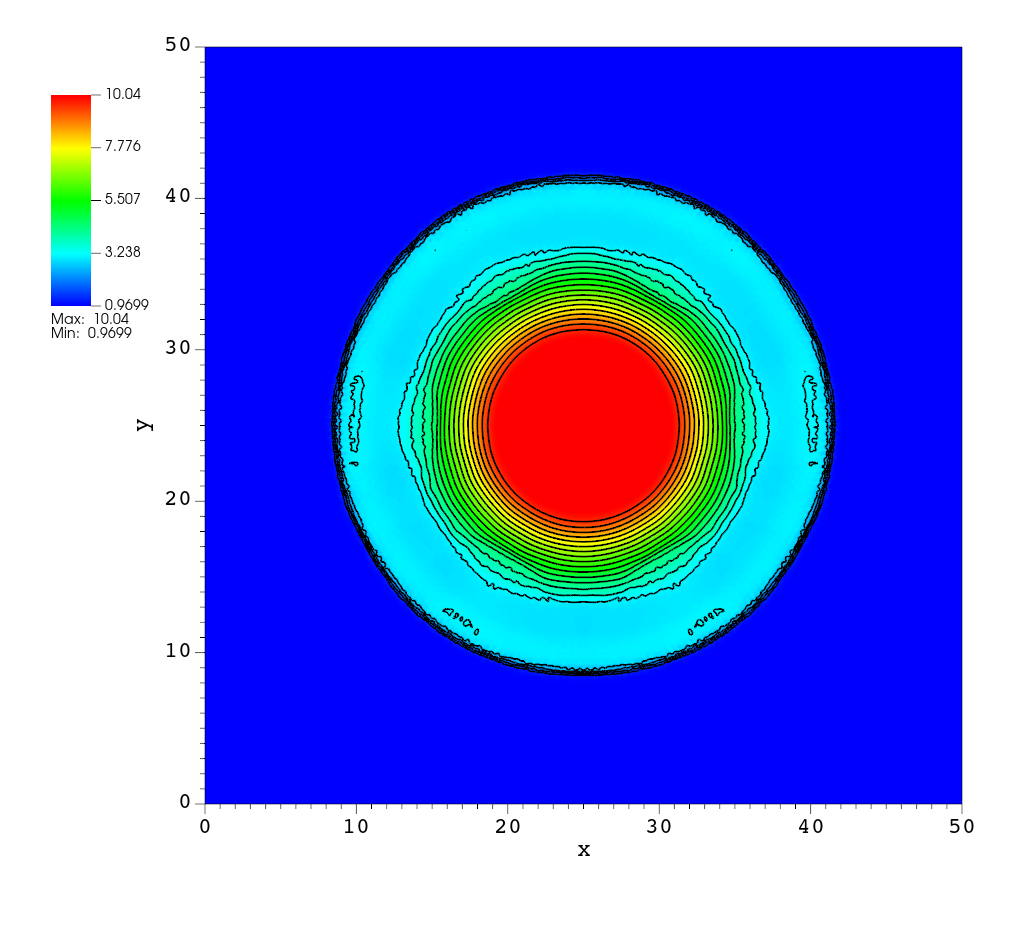}}}
\vskip5pt
\centerline{\subfigure[flag on elements at $t=0.447$]{\includegraphics[trim=1.8cm 2.25cm 1.1cm 1.2cm,clip,width=5.0cm]{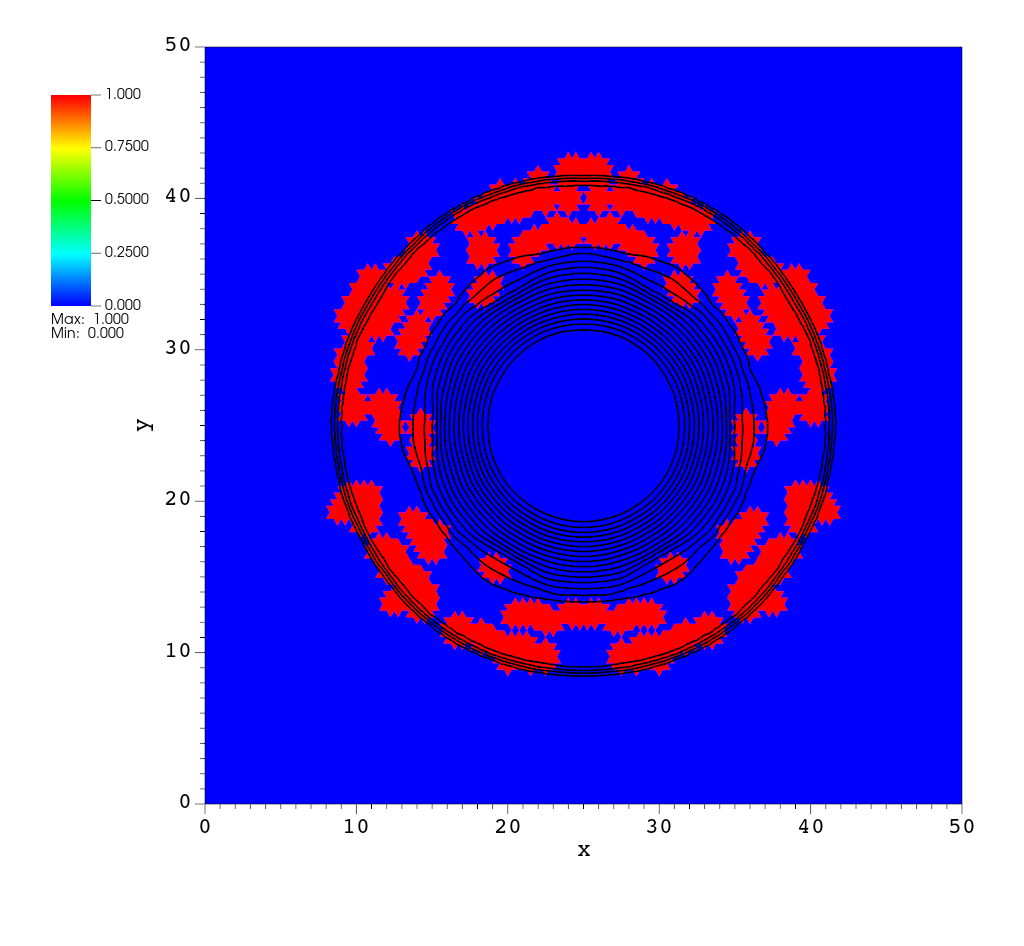}}\hspace*{0.15cm}
	\subfigure[flag on DoFs at $t=0.447$]{\includegraphics[trim=1.8cm 2.25cm 1.1cm 1.2cm,clip,width=5.0cm]{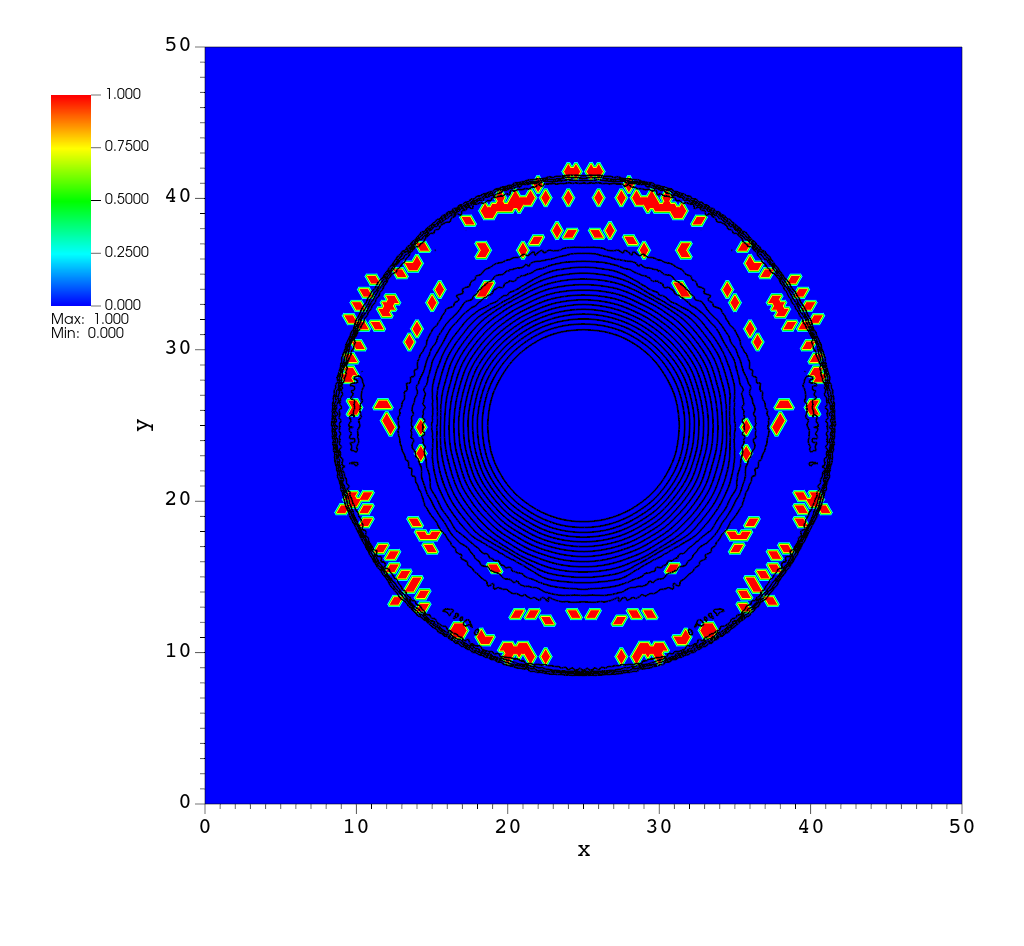}}}
\caption{\sf Example 5: Cell averages and point values of water depth (top row) and flags at the positions where the first-order scheme is used (bottom row) at $t=0.447$.\label{Ex5_h}}
\end{figure}

\begin{figure}[ht!]
\centerline{\subfigure[$\xbar{h}_E$ at $t=0.69$]{\includegraphics[trim=1.8cm 2.25cm 1.1cm 1.2cm,clip,width=5.0cm]{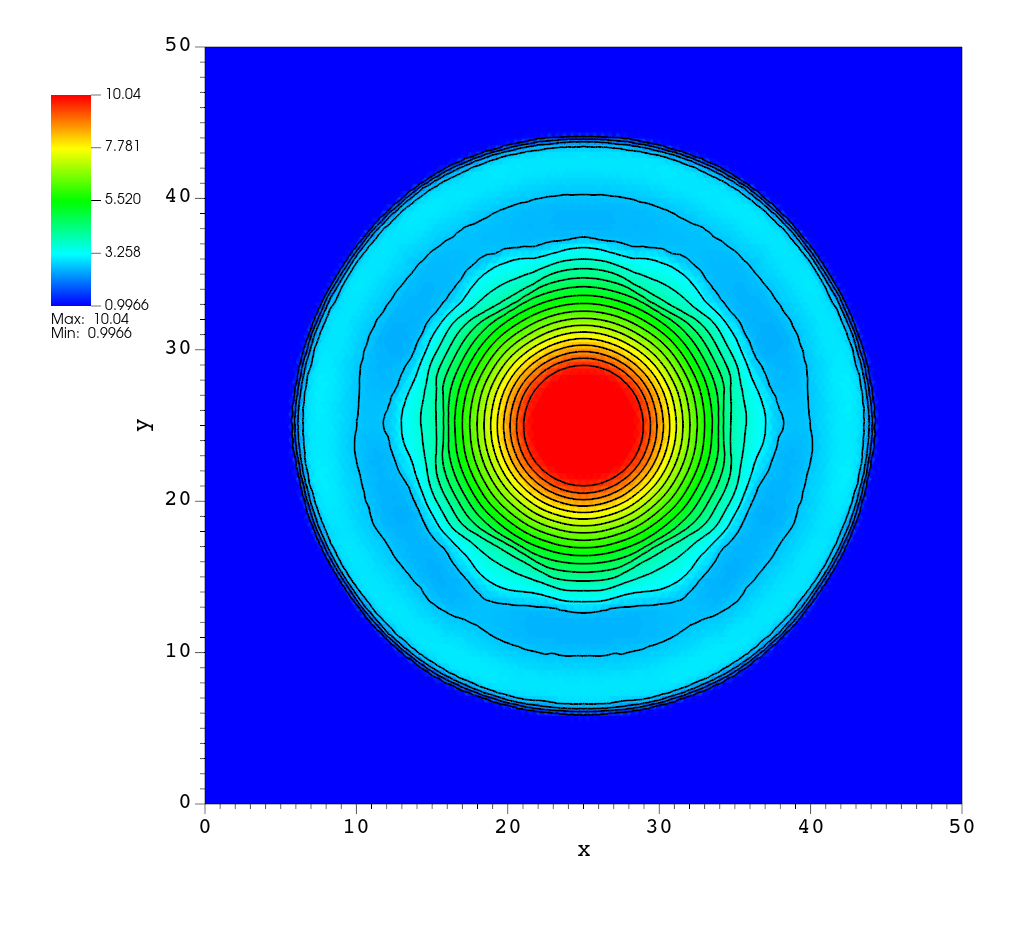}}\hspace*{0.15cm}
	\subfigure[$h_\sigma$ at $t=0.69$]{\includegraphics[trim=1.8cm 2.25cm 1.1cm 1.2cm,clip,width=5.0cm]{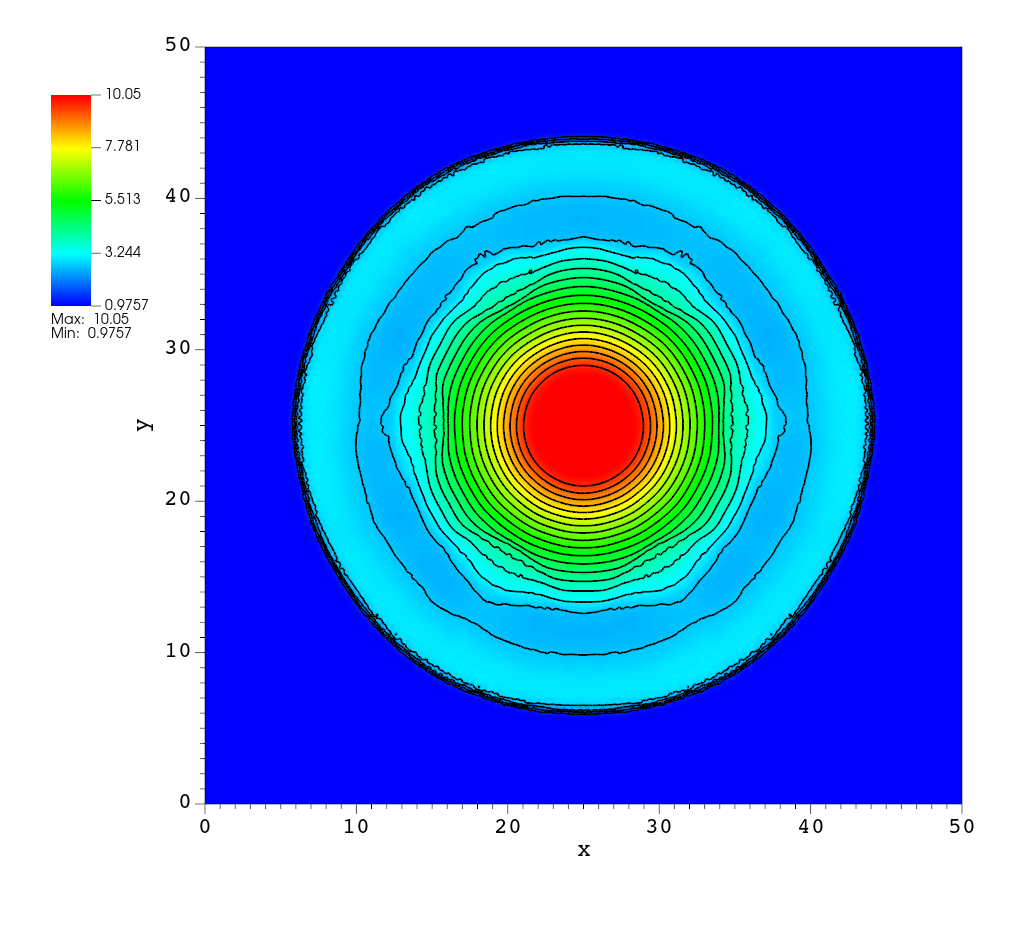}}}
\vskip5pt
\centerline{\subfigure[flag on elements at $t=0.69$]{\includegraphics[trim=1.8cm 2.25cm 1.1cm 1.2cm,clip,width=5.0cm]{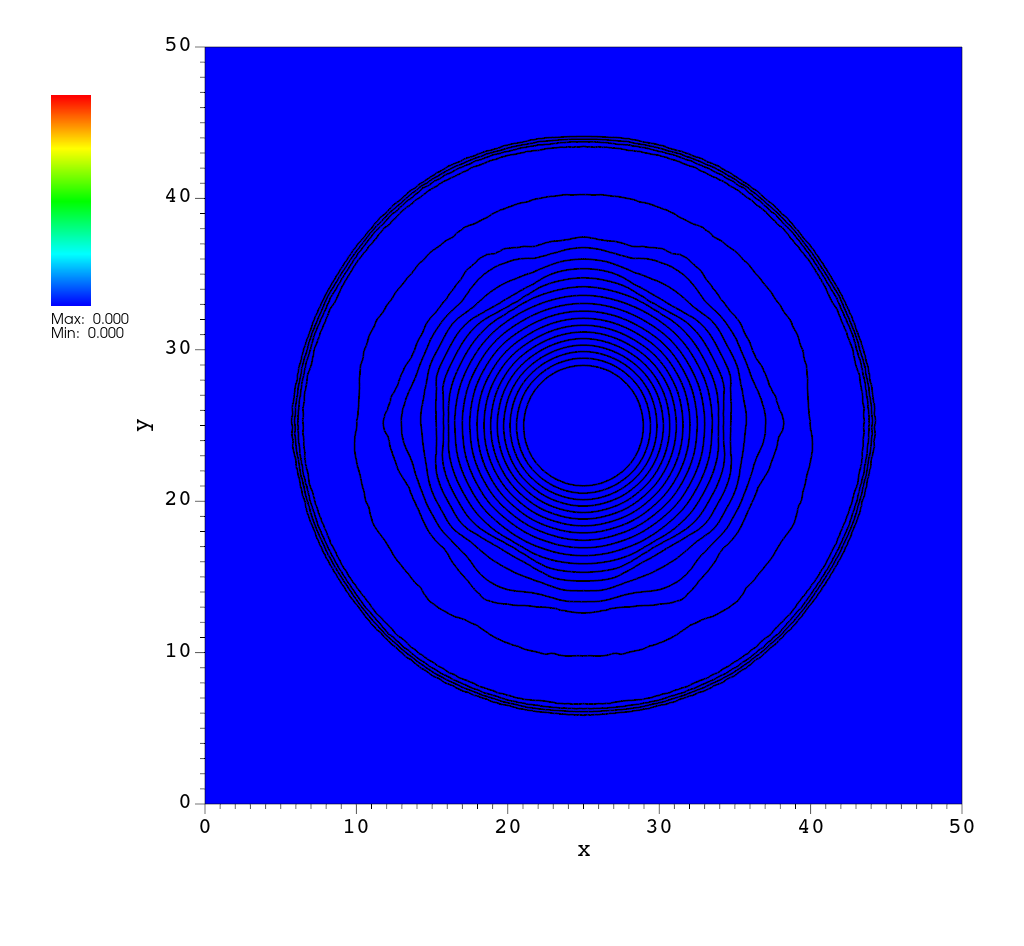}}\hspace*{0.15cm}
	\subfigure[flag on DoFs at $t=0.69$]{\includegraphics[trim=1.8cm 2.25cm 1.1cm 1.2cm,clip,width=5.0cm]{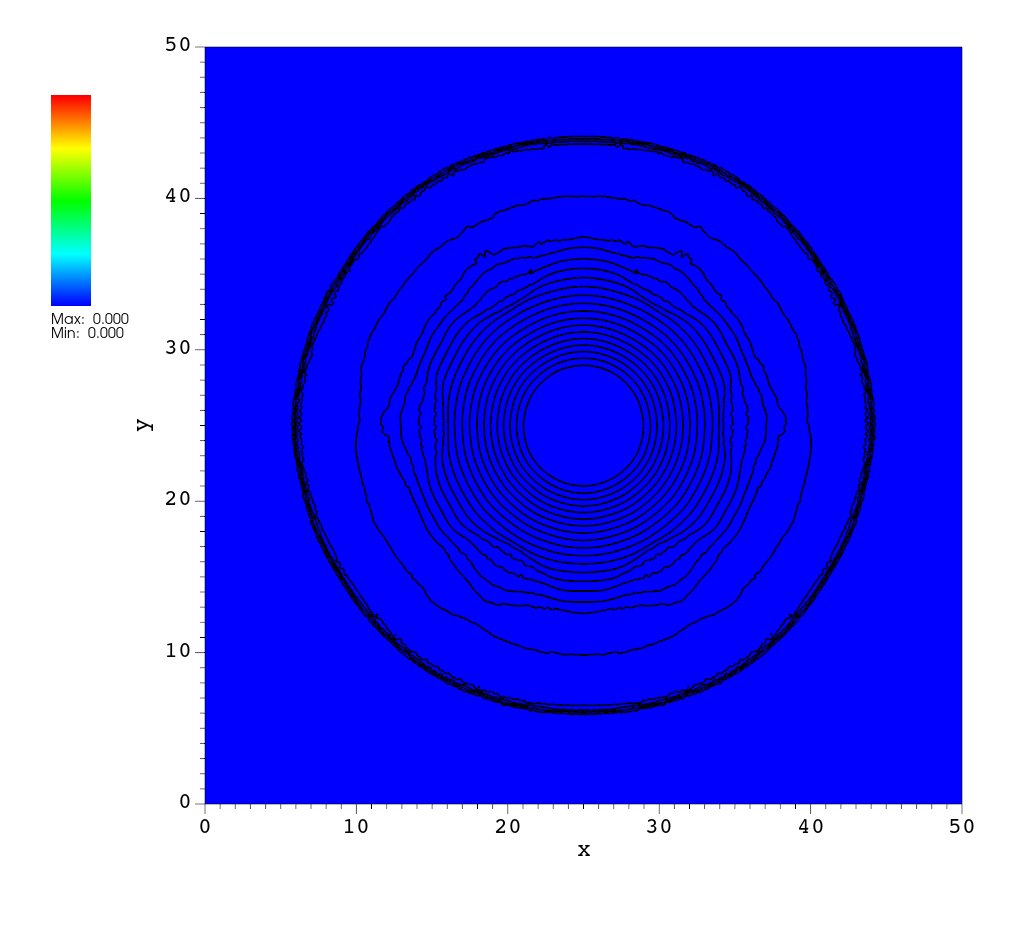}}}
\caption{\sf Example 5: Cell averages and point values of water depth (top row) and flags at the positions where the first-order scheme is used (bottom row) at $t=0.69$.\label{Ex5_h2}}
\end{figure}

\vla{In the bottom row of Figures \ref{Ex5_h} and \ref{Ex2_h1}, we also plot the regions where the first-order scheme is activated at the corresponding time. From the results at $t=0.447$, it is evident that the first-order schemes for average and point values are activated in a few elements around the discontinuities. We can also observe that at the final time $t=0.69$, the limiter is not active for both average and point values evolution, enough numerical dissipation has been produced by the scheme. The behavior of the limiter changes in time and in space, in Figure \ref{Ex5_flag2} we show how the number of limiters changes in time.}

\begin{figure}[ht!]
\centerline{\includegraphics[trim=0.02cm 0.04cm 0.5cm 0.2cm,clip,width=6.0cm]{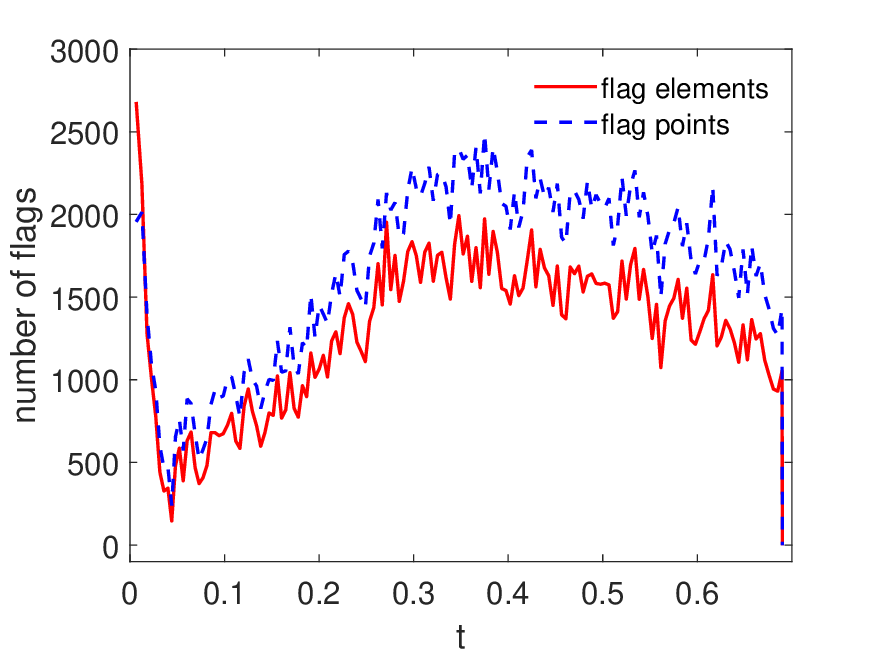}}
\caption{\sf Example 5: Number of flagged elements and points as a function of time $t$.\label{Ex5_flag2}}
\end{figure}

\vla{\subsection*{Example 6---Isobaric steady state}
In the sixth example, we consider a flat bottom topography (i.e., $Z(\mbf x)=0$) and the following initial condition:
\begin{equation*}
  (h,\bm\nu,\theta)^T(\mbf x,0)=\left\{
  \begin{aligned}
  &(2e^{-0.05\Vert\mbf x-30\Vert^2},\mbf 0, e^{0.1\Vert\mbf x-30\Vert^2})^T,&&\Vert\mbf x-30\Vert^2\leq 16\\
  &(2e^{-0.8},\mbf 0, e^{1.6})^T, && \Vert\mbf x-30\Vert^2>16,
  \end{aligned}\right.
\end{equation*}
which satisfies the isobaric steady state \eqref{1.2b}. We compute the solution up to $t=2$ using triangular meshes with $3718$ elements and $7597$ DoFs over the domain $[0,40]\times[0,40]$.  From \cref{alg:Gaussian_integrals} and \cref{prop:WB}, we know that the PAMPA scheme is WB if the 3-point Gauss-Lobatto quadrature rule is applied to compute the edge integral \eqref{2.7} on the local plateau bottom region. In contrast, the non-WB scheme is set directly using the 5-point Gauss--Legendre quadrature rule for the edge integral. The results, reported in Tables \ref{Ex6:tab1} and \ref{Ex6:tab2}, clearly demonstrate that the WB PAMPA scheme preserves the isobaric equilibrium within machine accuracy, whereas the non-WB scheme fails. This further confirms the effectiveness of the adapted quadrature rule selection procedure outlined in Algorithm \ref{alg:Gaussian_integrals}.
 
\begin{table}[ht!]
\caption{\sf Example 6: $L^1$- and $L^\infty$-errors in $\mbf u_E$ for the isobaric state obtained by WB PAMPA scheme and its non-WB counterpart.\label{Ex6:tab1}}
\begin{center}
\begin{tabular}{|c| c c| c c| }\hline
\multicolumn{1}{|c|}{\multirow{2}{*}{Schemes}}  &\multicolumn{2}{c|}{$h$} &\multicolumn{2}{c|}{$h\theta$} \\ \cline{2-5}
  & $L^1$-error   & $L^\infty$-error  & $L^1$-error  &$L^\infty$-error  \\ \hline
 \multicolumn{1}{|c|}{\multirow{1}{*}{WB}}  &$2.52\times10^{-18}$ & $4.44\times10^{-16}$ & $2.32\times10^{-17}$ & $9.77\times10^{-15}$  \\  \hline
 \multicolumn{1}{|c|}{\multirow{1}{*}{non-WB}}  &$3.62\times10^{-5}$ & $4.33\times10^{-4}$ & $2.42\times10^{-5}$& $5.05\times10^{-3}$\\ \hline
\end{tabular}
\end{center}
\end{table}

\begin{table}[ht!]
\caption{\sf Example 6: Same as in Table \ref{Ex6:tab1} but for point values $\mbf u_\sigma$.\label{Ex6:tab2}}
\begin{center}
\begin{tabular}{|c| c c| c c| }\hline
\multicolumn{1}{|c|}{\multirow{2}{*}{Schemes}}  &\multicolumn{2}{c|}{$h$} &\multicolumn{2}{c|}{$h\theta$} \\ \cline{2-5}
  & $L^1$-error   & $L^\infty$-error  & $L^1$-error  &$L^\infty$-error  \\ \hline
 \multicolumn{1}{|c|}{\multirow{1}{*}{WB}}  &$2.33\times10^{-17}$ & $4.44\times10^{-16}$ & $1.49\times10^{-16}$ & $1.78\times10^{-15}$  \\  \hline
 \multicolumn{1}{|c|}{\multirow{1}{*}{non-WB}}  &$5.64\times10^{-5}$ & $1.06\times10^{-3}$ & $2.42\times10^{-5}$& $5.05\times10^{-3}$\\ \hline
\end{tabular}
\end{center}
\end{table}

To further demonstrate the effectiveness of the WB PAMPA scheme in handling isobaric equilibrium, we study the propagation of a small perturbation. To this end, a small perturbation of $0.005e^{-0.2\Vert\mbf x-20\Vert^2}$ is superimposed onto the initial water depth. Figure \ref{Ex6_h1} displays the difference between the numerically computed and background (steady-state) water depth averages at time $t=4.3$, when the perturbation begins interacting with the non-zero water depth and temperature region. Outside this region, the WB PAMPA and non-WB schemes exhibit similar behavior. However, as the wavefront enters, the non-WB scheme produces unphysical oscillations, whereas the WB PAMPA scheme remains free of spurious artifacts ahead of the wavefront. This distinction is also evident at time $t=6$, as shown in Figure \ref{Ex6_h2}. These differences persist despite the use of a fine mesh, highlighting the robustness of the WB PAMPA scheme. 

\begin{figure}[ht!]
\centerline{\subfigure[WB PAMPA]{\includegraphics[trim=0.2cm 1.5cm 2.0cm 1.0cm,clip,width=5.0cm]{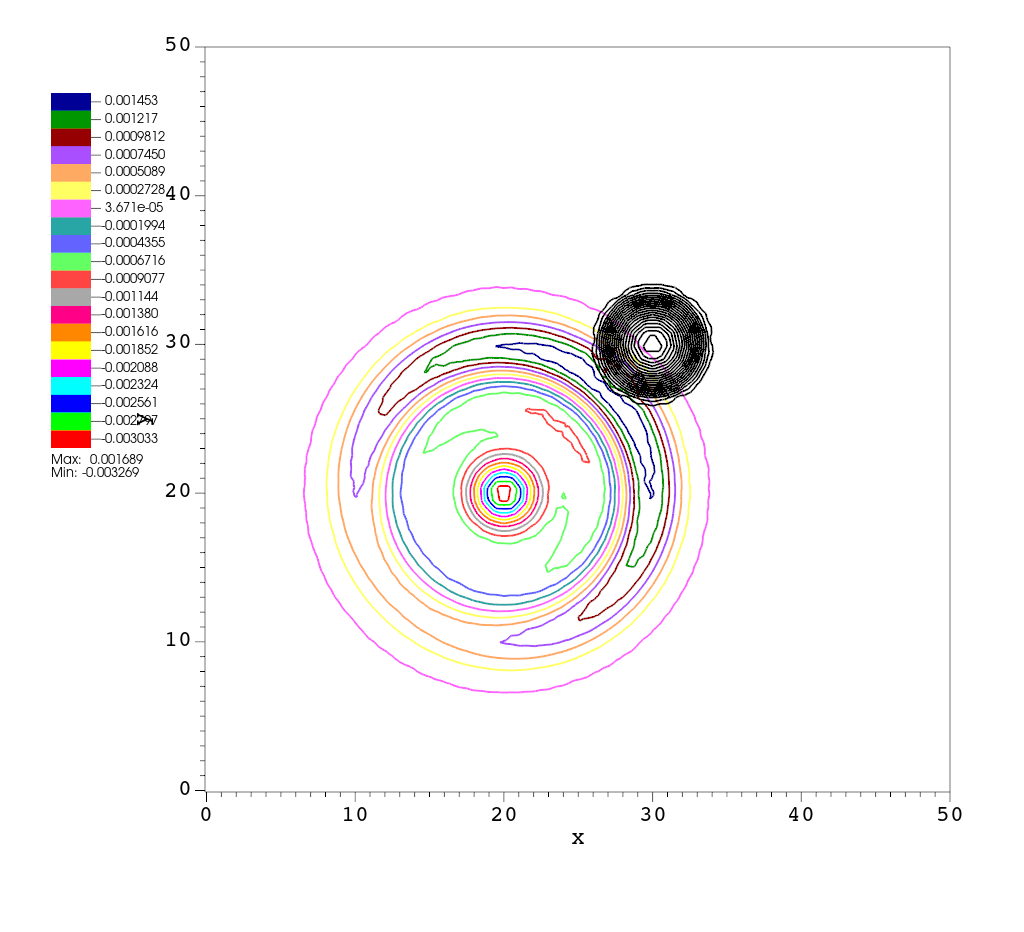}}\hspace*{0.5cm}
	\subfigure[non WB]{\includegraphics[trim=0.2cm 1.5cm 2.0cm 1.0cm,clip,width=5.0cm]{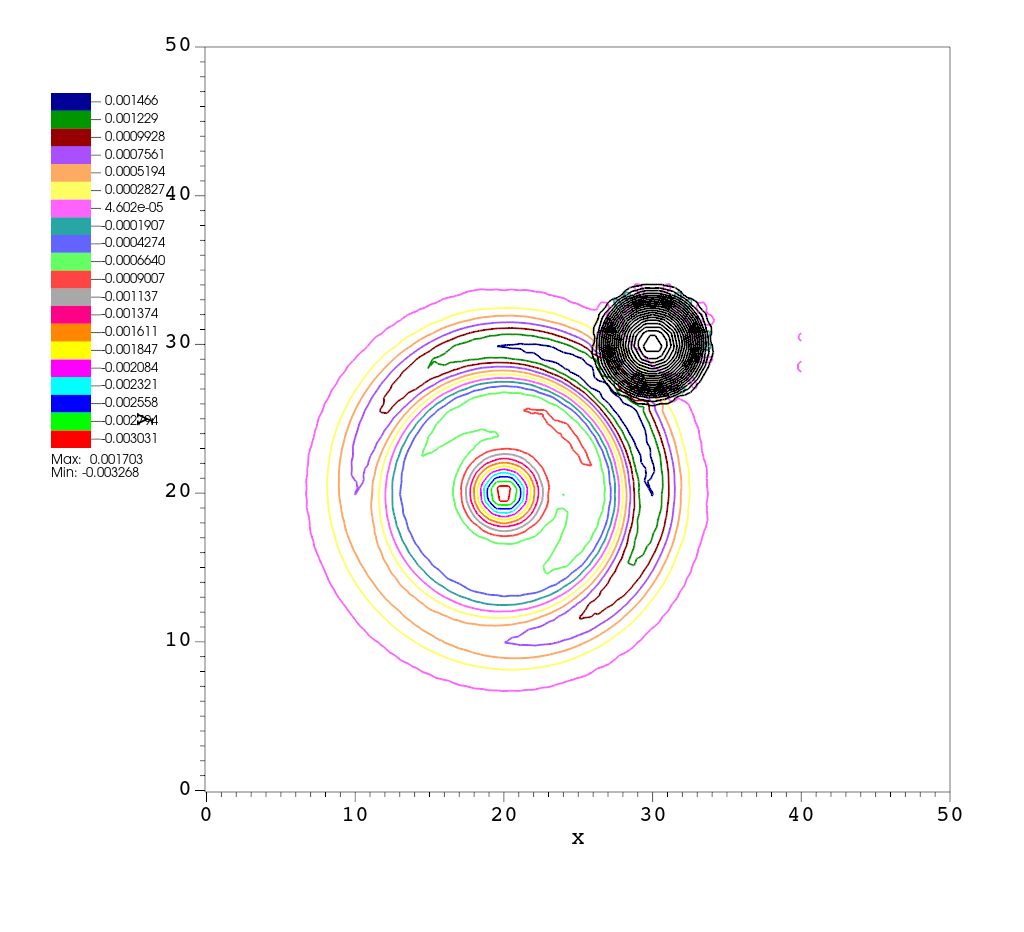}}}
\caption{\sf Example 6: Perturbation in the cell averages of water depth at $t=4.3$ computed by the WB PAMPA (left) and non well-balanced (right) schemes. The black contour lines stand for the initial non-zero values of $h$.\label{Ex6_h1}}
\end{figure}

\begin{figure}[ht!]
\centerline{\subfigure[WB PAMPA]{\includegraphics[trim=0.2cm 1.5cm 2.0cm 1.0cm,clip,width=5.0cm]{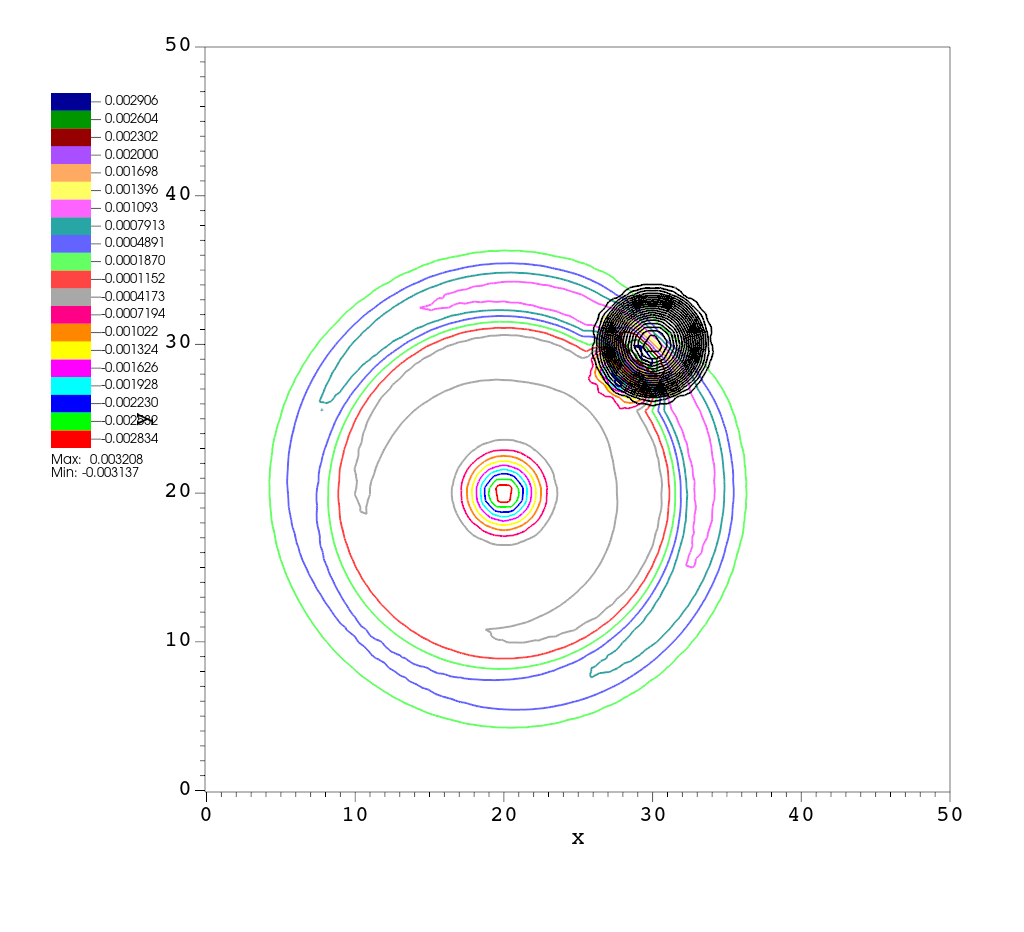}}\hspace*{0.5cm}
	\subfigure[non WB]{\includegraphics[trim=0.2cm 1.5cm 2.0cm 1.0cm,clip,width=5.0cm]{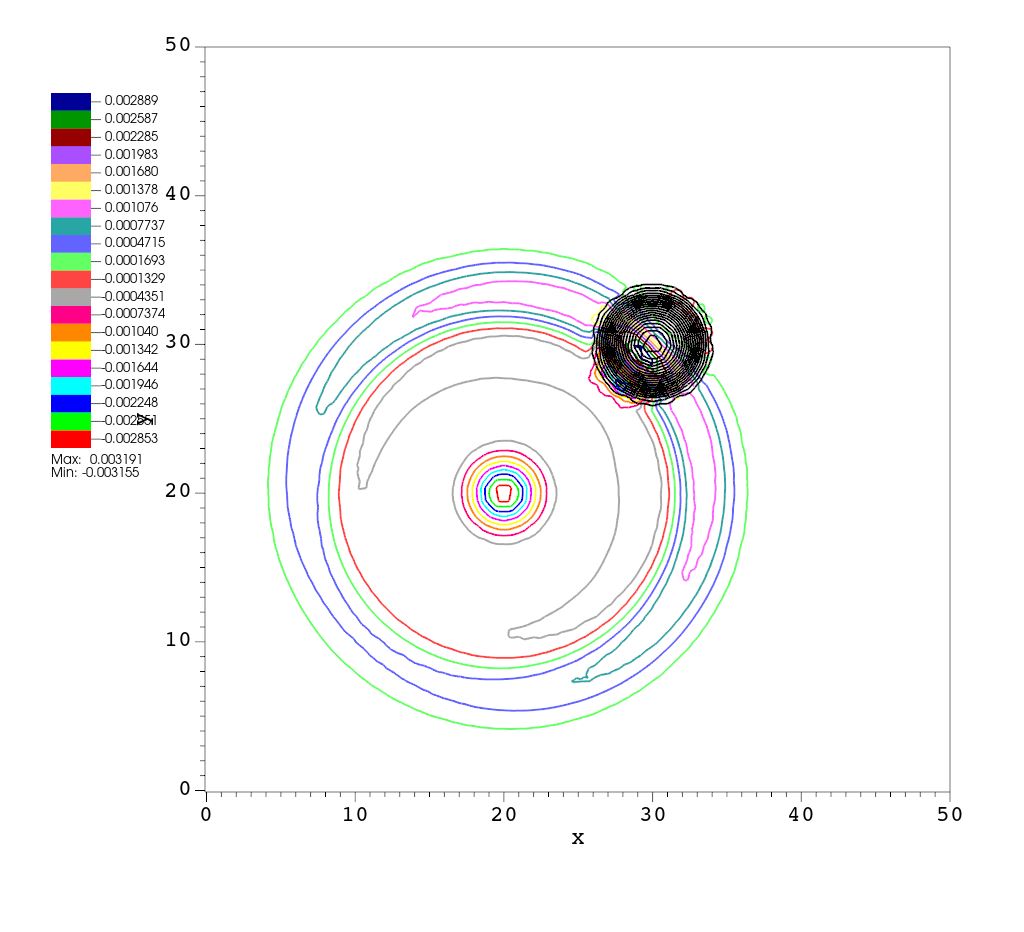}}}
\caption{\sf Example 6: Same as in Figure \ref{Ex6_h1} but for $t=6$.\label{Ex6_h2}}
\end{figure}

\subsection*{Example 7---Two isolated ``lake at rest'' states}
In the seventh example, taken from  \cite{CKL}, we consider the following initial data consisting of two ``lake at rest'' states connected through the temperature jump:
\begin{equation*}
  (h,\bm\nu,\theta)^T(\mbf x,0)=\left\{
  \begin{aligned}
  &(3-Z,\mbf 0, \frac{4}{3})^T,& &\Vert\mbf x\Vert\leq0.5,\\
  &(2-Z,\mbf 0,3)^T,& &\Vert\mbf x\Vert>0.5,
  \end{aligned}\right.
\end{equation*}
where $Z$ is a non-flat bottom topography given by
\begin{equation*}
  Z(\mbf x)=\left\{
  \begin{aligned}
  &0.5e^{-100\Vert\mbf x+0.5\Vert^2},& &x<0,\\
  &0.6e^{-100\Vert\mbf x-0.5\Vert^2},& &x>0.
  \end{aligned}\right.
\end{equation*}

The solution is computed using the WB PAMPA scheme on a triangular mesh with $9524$ elements and $18761$ DoFs up to a final time of $t=0.12$. Figure \ref{Ex7_thetap} presents the solutions of temperature and pressure, where we define $\theta_E:=\xbar{(h\theta)}_E/\xbar h_E$ and $p_E=\xbar{h}_E\xbar{(h\theta)}_E$. As shown, the results obtained with fewer computational cells are highly consistent with those computed using the complex interface tracking technique in \cite{CKL} and outperform their version without this technique.

\begin{figure}[ht!]
\centerline{\subfigure[$\theta_E$]{\includegraphics[trim=1.8cm 2.25cm 1.1cm 1.2cm,clip,width=5.0cm]{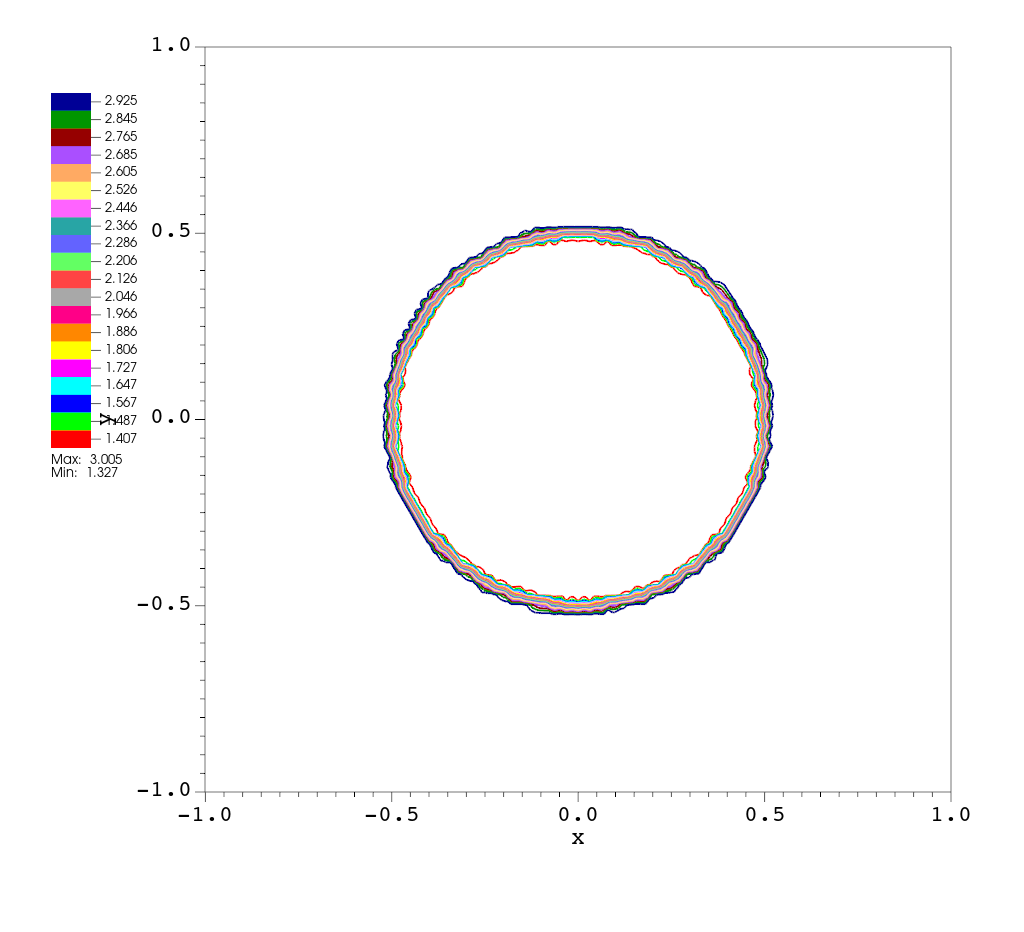}}\hspace*{0.15cm}
	\subfigure[$\theta_\sigma$]{\includegraphics[trim=1.8cm 2.25cm 1.1cm 1.2cm,clip,width=5.0cm]{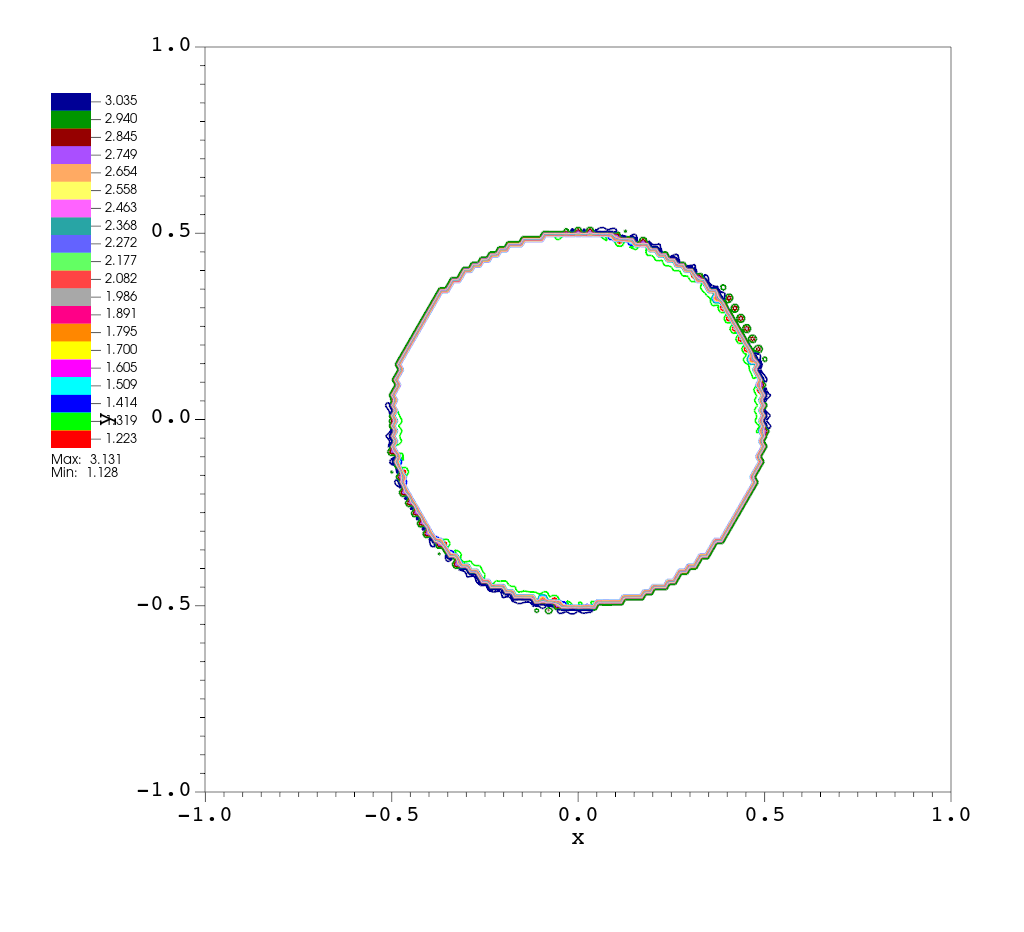}}}
\vskip5pt
\centerline{\subfigure[$p_E$]{\includegraphics[trim=1.8cm 2.25cm 1.1cm 1.2cm,clip,width=5.0cm]{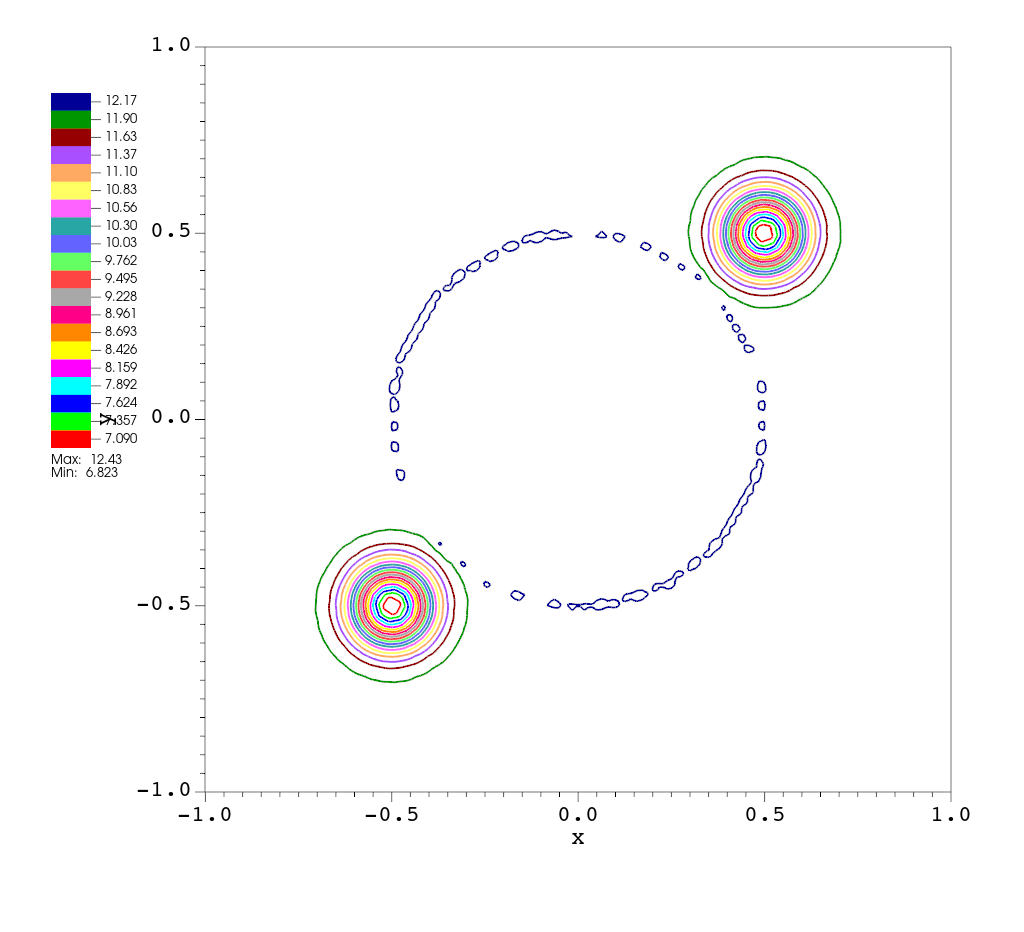}}\hspace*{0.15cm}
	\subfigure[$p_\sigma$]{\includegraphics[trim=1.8cm 2.25cm 1.1cm 1.2cm,clip,width=5.0cm]{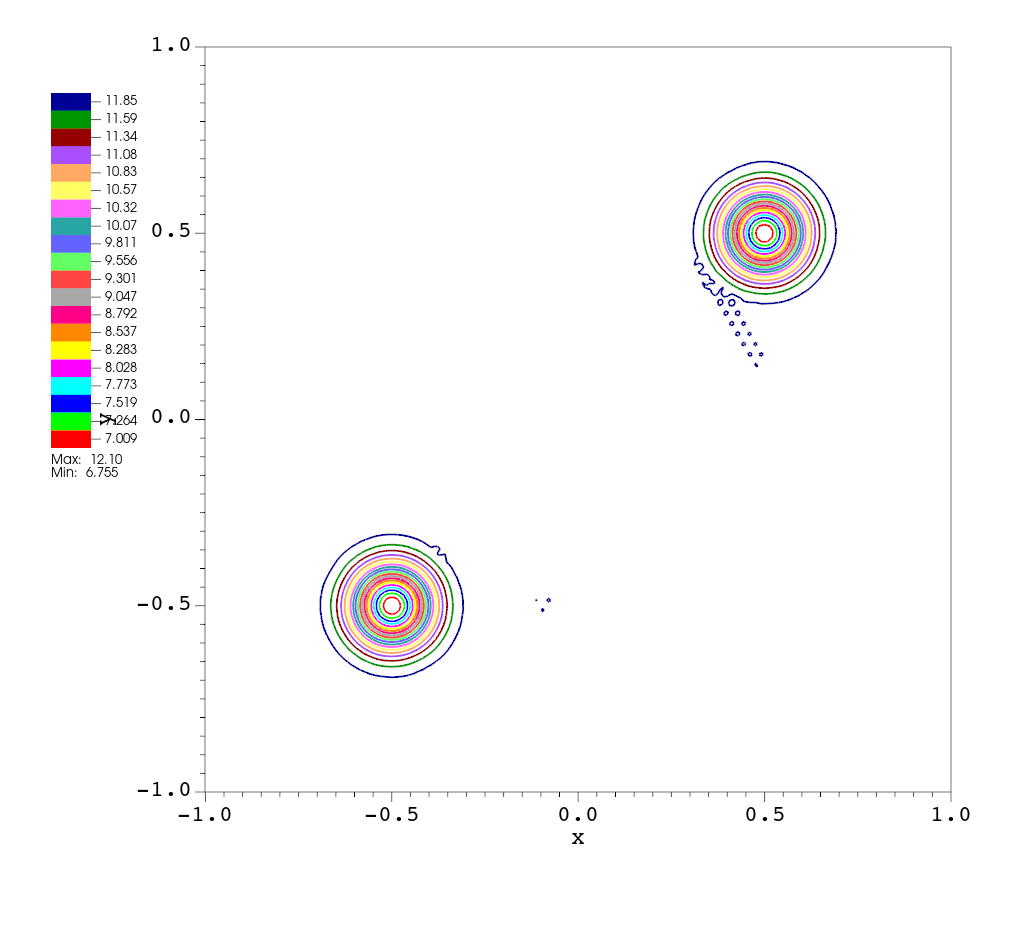}}}
\caption{\sf Example 7: Solutions of temperature (top row) and pressure (bottom row) at $t=0.12$.\label{Ex7_thetap}}
\end{figure}

\subsection*{Example 8---Small perturbation of two isolated ``lake at rest'' states}
In the eighth example, taken from \cite{CKL}, we use the same setting as in Example 7 and take the following initial condition:
\begin{equation*}
  (h,\bm\nu,\theta)^T(\mbf x,0)=\left\{
  \begin{aligned}
  &(3.1-Z,\mbf 0, \frac{4}{3})^T,& &0.1\leq\Vert\mbf x\Vert\leq0.3,\\
  &(3-Z,\mbf 0, \frac{4}{3})^T,& &0.3<\Vert\mbf x\Vert\leq0.5,\\
  &(2-Z,\mbf 0,3)^T,& &\Vert\mbf x\Vert>0.5,
  \end{aligned}\right.
\end{equation*}
where a perturbation is added inside a small annulus near the center of the computational domain. The results are presented in Figure \ref{Ex8_thetap}, where the appearance of numerical oscillations due to the sensitive of setting the MOOD criteria. A further study is needed to reduce these oscillations.

\begin{figure}[ht!]
\centerline{\subfigure[$\theta_E$]{\includegraphics[trim=1.8cm 2.25cm 1.1cm 1.2cm,clip,width=5.0cm]{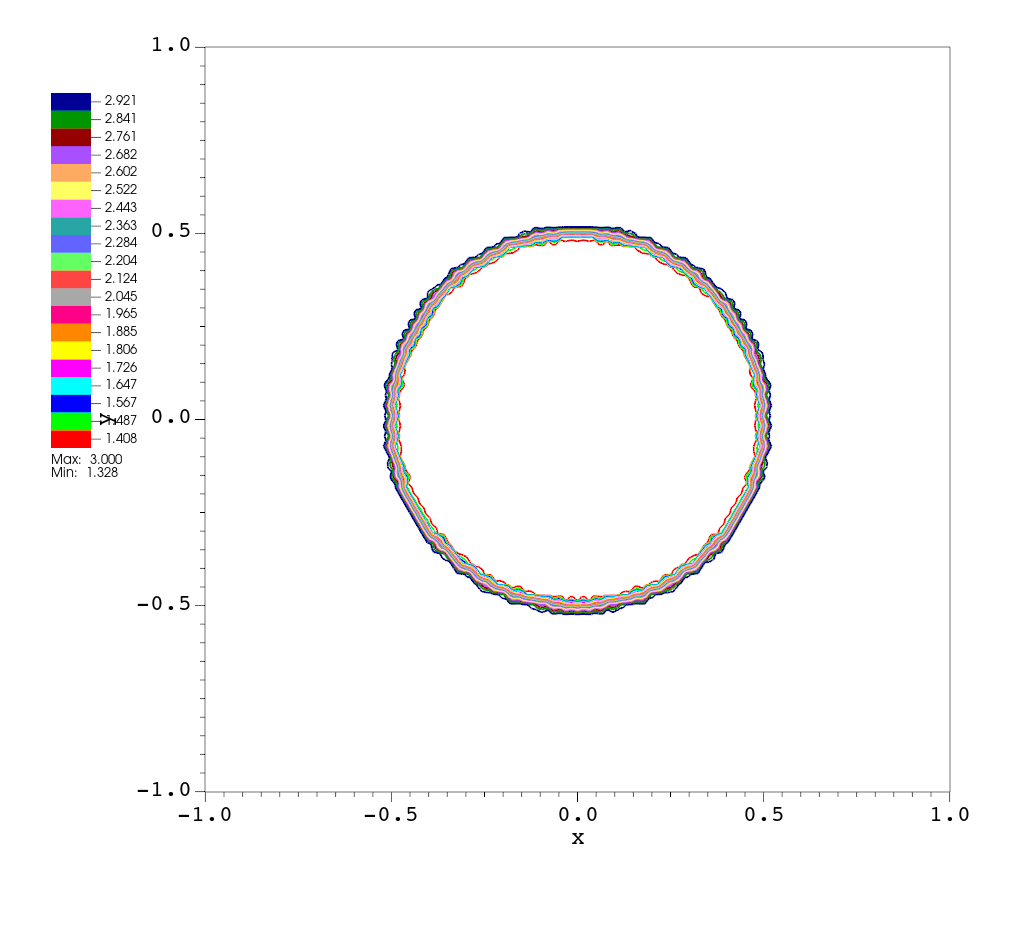}}\hspace*{0.15cm}
	\subfigure[$\theta_\sigma$]{\includegraphics[trim=1.8cm 2.25cm 1.1cm 1.2cm,clip,width=5.0cm]{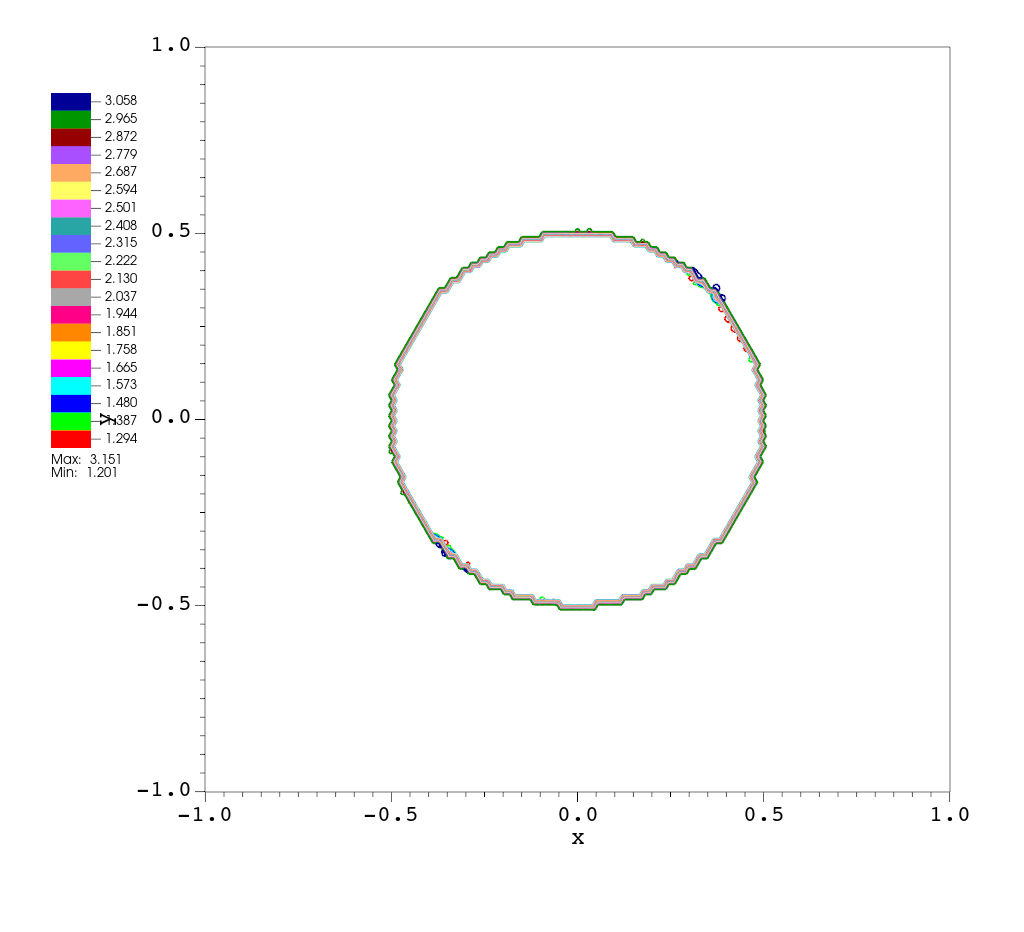}}}
\vskip5pt
\centerline{\subfigure[$p_E$]{\includegraphics[trim=1.8cm 2.25cm 1.1cm 1.2cm,clip,width=5.0cm]{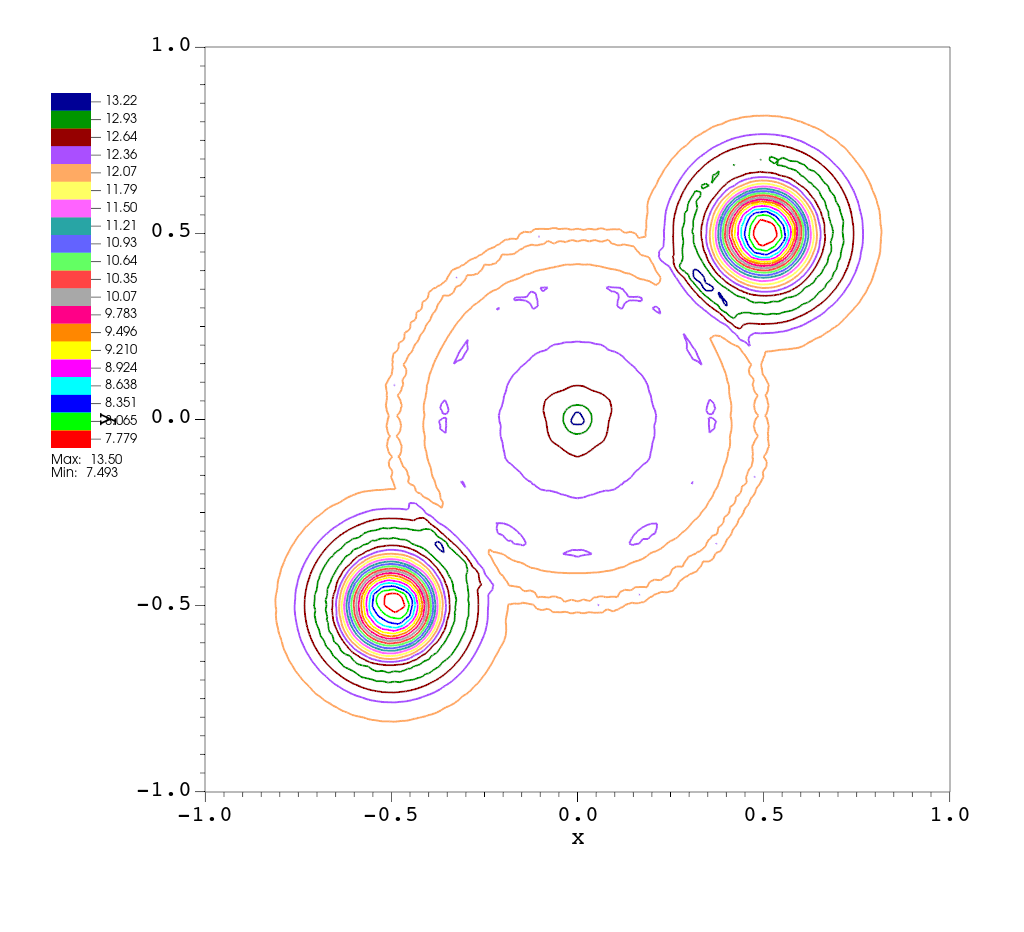}}\hspace*{0.15cm}
	\subfigure[$p_\sigma$]{\includegraphics[trim=1.8cm 2.25cm 1.1cm 1.2cm,clip,width=5.0cm]{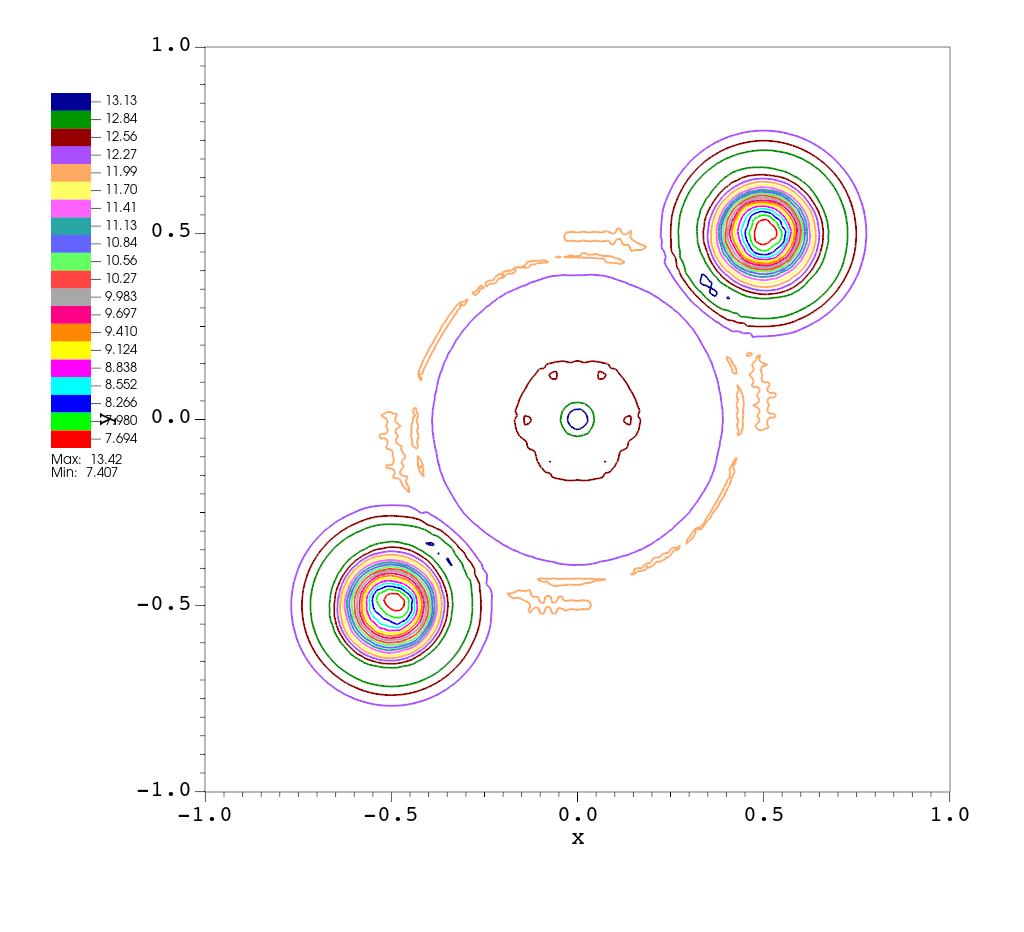}}}
\caption{\sf Example 8: Solutions of temperature (top row) and pressure (bottom row) at $t=0.05$.\label{Ex8_thetap}}
\end{figure}

\subsection*{Example 9---Radial dam break with non-constant temperature}
In the final example, we consider a radial dam break problem with a flat bottom topography (i.e., $Z(\mbf x)=0$). The initial condition is given by (taken from \cite{CKL}):
\begin{equation*}
  (h,\bm\nu,\theta)^T(\mbf x,0)=\left\{
  \begin{aligned}
  &(2,\mbf 0, 1)^T,& &\Vert\mbf x\Vert\leq0.5,\\
  &(1,\mbf 0, 1.5)^T,& &\Vert\mbf x\Vert>0.5.
  \end{aligned}\right.
\end{equation*}
We compute the solution using the WB PAMPA scheme on a triangular mesh with $9524$ elements and $18761$ DoFs until a final time $t=0.15$. The results for temperature and pressure are plotted in Figure \ref{Ex9_thetap}, which comparable with those reported in \cite{CKL}.
\begin{figure}[ht!]
\centerline{\subfigure[$\theta_E$]{\includegraphics[trim=1.8cm 2.25cm 1.1cm 1.2cm,clip,width=5.0cm]{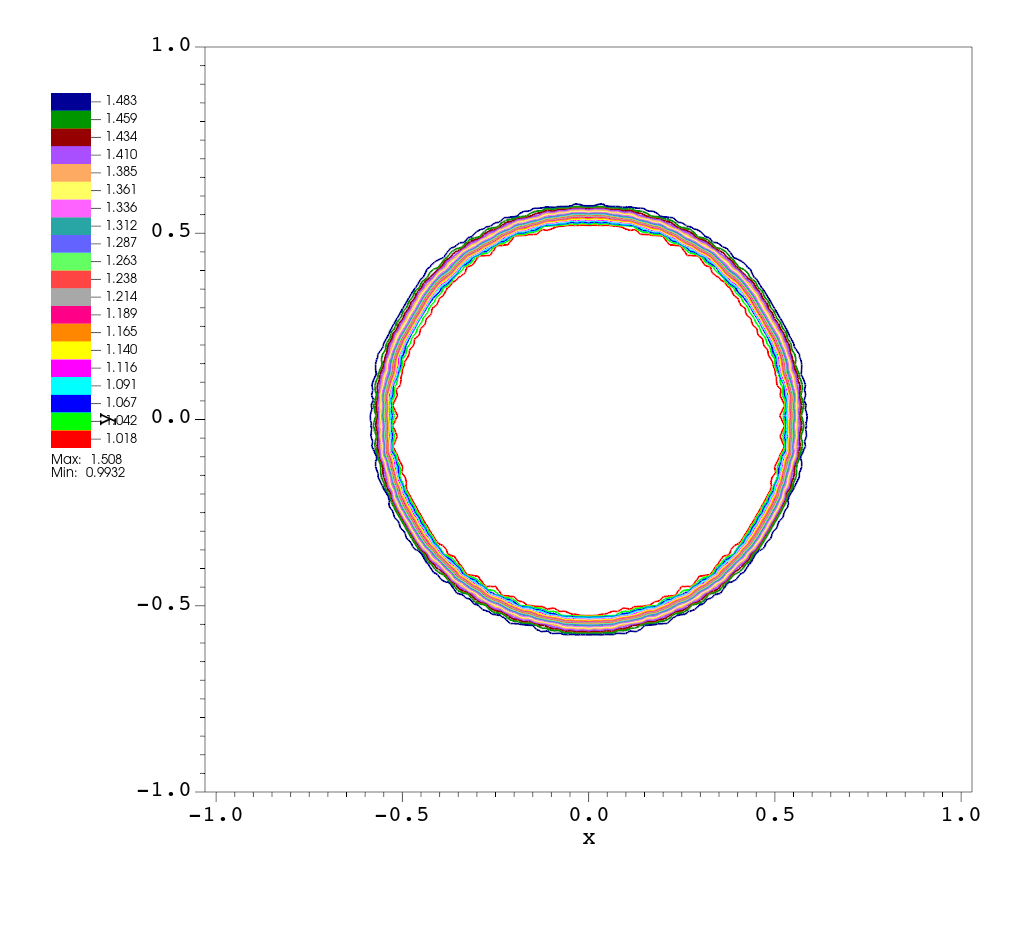}}\hspace*{0.15cm}
	\subfigure[$\theta_\sigma$]{\includegraphics[trim=1.8cm 2.25cm 1.1cm 1.2cm,clip,width=5.0cm]{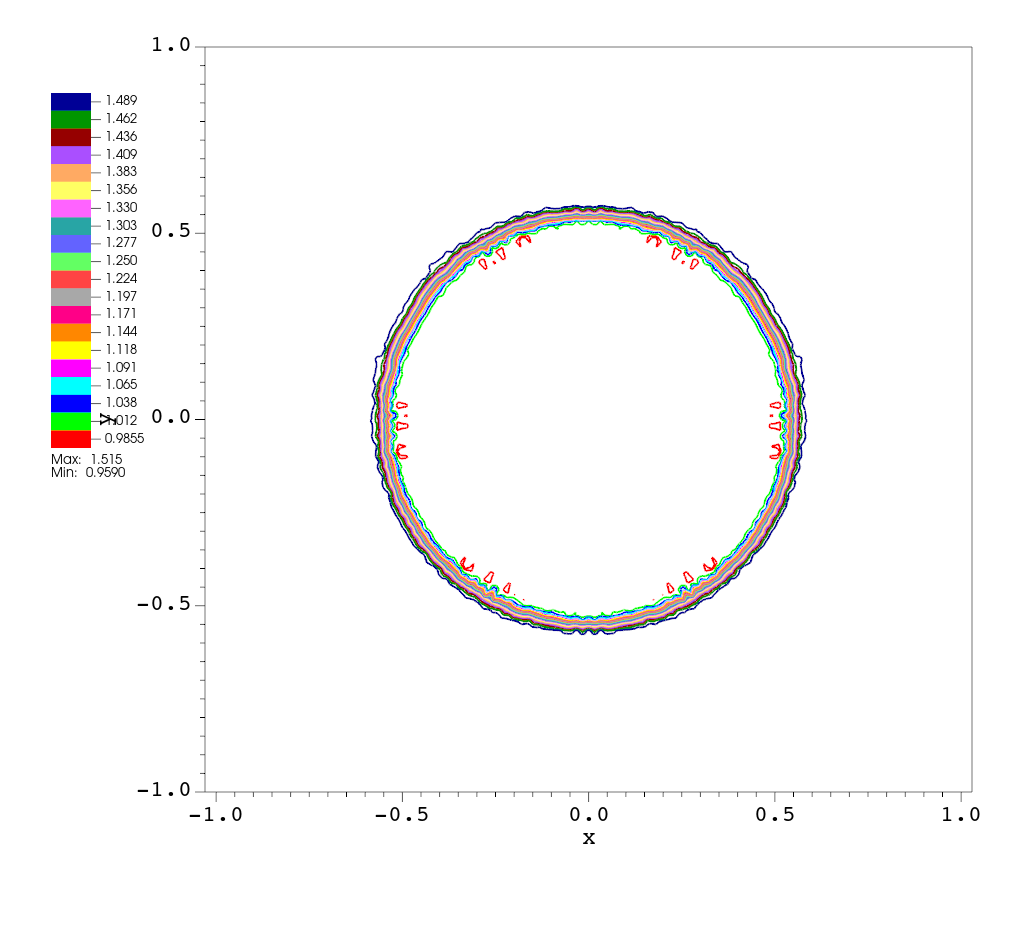}}}
\vskip5pt
\centerline{\subfigure[$p_E$]{\includegraphics[trim=1.8cm 2.25cm 1.1cm 1.2cm,clip,width=5.0cm]{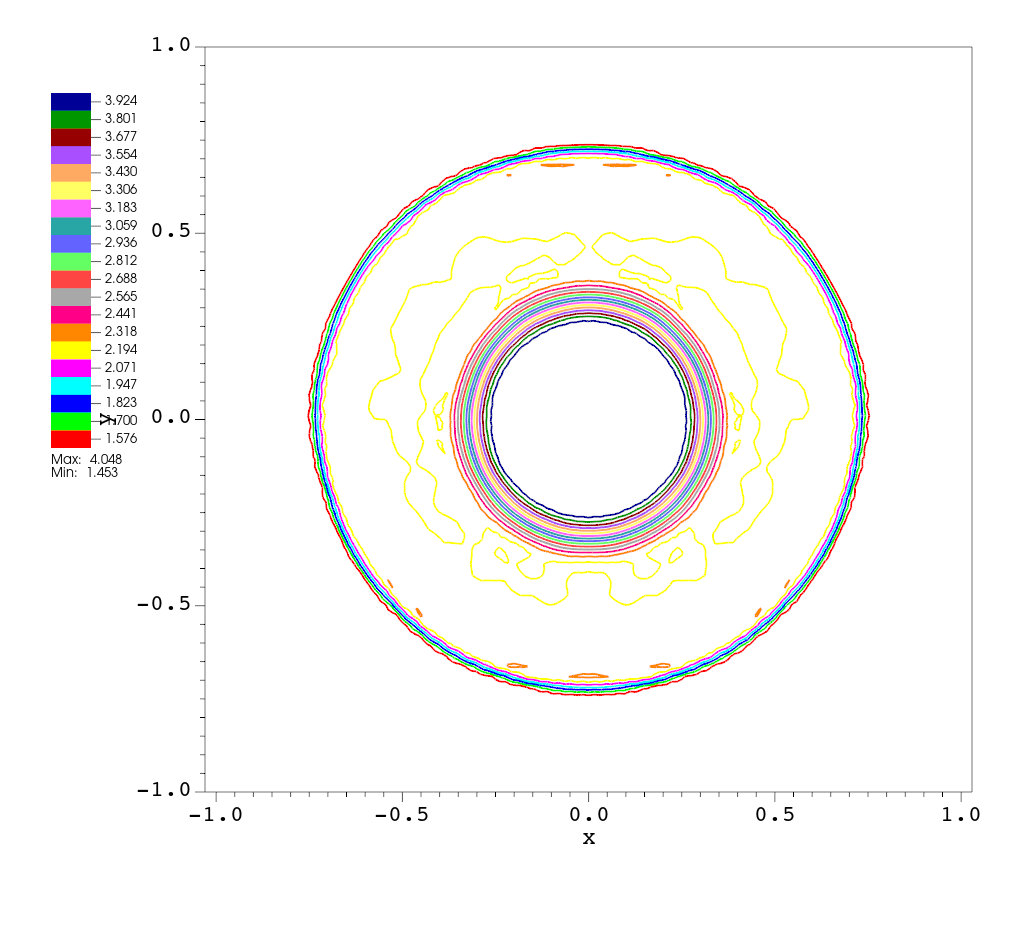}}\hspace*{0.15cm}
	\subfigure[$p_\sigma$]{\includegraphics[trim=1.8cm 2.25cm 1.1cm 1.2cm,clip,width=5.0cm]{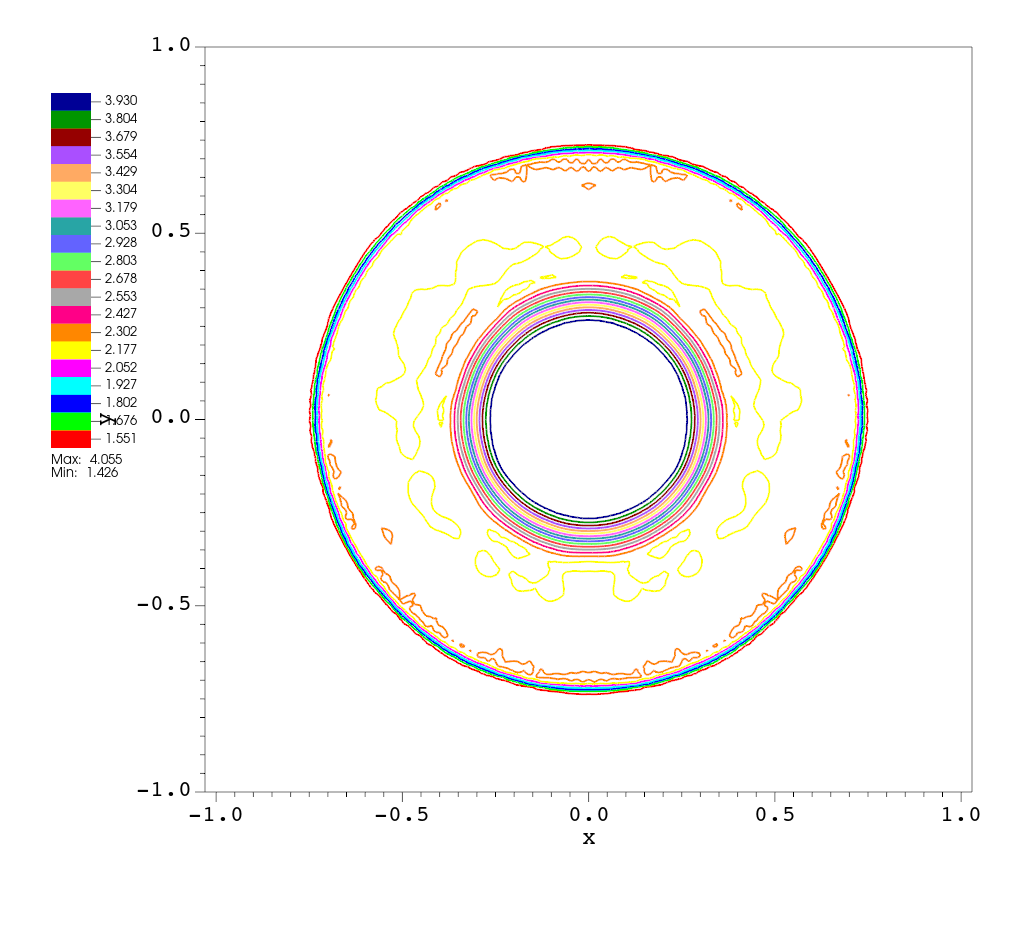}}}
\caption{\sf Example 9: Solutions of temperature (top row) and pressure (bottom row) at $t=0.15$.\label{Ex9_thetap}}
\end{figure}
}

\section{Conclusion}\label{sec5}
In this paper, we have presented a new well-balanced PAMPA method on unstructured triangular meshes for solving the two-dimensional \vla{shallow water equations with temperature gradients}. In the proposed method, the solution is a globally continuous representation of cell average within each triangular element, and point values located on the boundaries of these elements. The conservative property is recovered by evolving the cell averages in a conservative form, while there is great flexibility in handling the evolution of the point values, which is governed by a non-conservative formulation of the same PDEs. \vla{We have demonstrated that using standard conservative and primitive variables in the non-conservative form can preserve only the ``lake at rest'' equilibrium. In order to preserve the isobaric steady state simultaneously, we introduce pressure-momentum-temperature variables.} We have shown that the studied scheme based on the newly introduced variables are well-balanced. Additionally, we have demonstrated that the proposed schemes are positivity-preserving by using the MOOD paradigm and the positivity-preserving first-order schemes. Several numerical examples with smooth and discontinuous problems have been tested to illustrate the performance of the proposed well-balanced PAMPA method. The results are as expected. 

In future work, we plan to extend the proposed PAMPA method to include higher-order schemes for enhance accuracy, \bla{a-priori monolithic convex limiting approaches to further improve efficiency based on the one-dimensional work \cite{BP_PAMPA_1D}}, and applications in more realistic and complex hyperbolic models, such as the thermal rotating shallow water equations. In this model, \vla{Coriolis forces are introduced, leading to a new class of equilibrium of particular interest---the so-called (thermo) geostrophic equilibrium. Accurately capturing these steady states is both physically significant and computationally challenging}.

\section*{Acknowledgement}
I am deeply grateful to Prof. R\'{e}mi Abgrall for our fruitful discussions on the high-order spatial discretizations for non-conservative formulations. The author also sincerely thank the two anonymous Reviewers for their constructive comments and very careful reading of the manuscript.

\bibliographystyle{siamplain}
\bibliography{reference_YL}
\end{document}

%% file: ex_shared.tex
\usepackage{lipsum}
\usepackage{amsfonts}
\usepackage{graphicx,subfigure}
\usepackage{epstopdf}
\usepackage{algorithmic}
\usepackage{xr-hyper}
\usepackage{cleveref}
\usepackage{nicefrac}
\usepackage{bm}
\usepackage{bbm}
\usepackage{amssymb}
\usepackage{amsmath}
\usepackage{mathabx}
\usepackage{mathrsfs}
\usepackage{mathtools}
\usepackage{empheq}
\usepackage{color}
\usepackage[normalem]{ulem}
\usepackage{stmaryrd}
\usepackage{multirow}
\usepackage[colorinlistoftodos,prependcaption,textsize=tiny]{todonotes}
\usepackage{extarrows}
\usepackage[]{hyperref}

\ifpdf
  \DeclareGraphicsExtensions{.eps,.pdf,.png,.jpg}
\else
  \DeclareGraphicsExtensions{.eps}
\fi

\def\softd{{\leavevmode\setbox1=\hbox{d}%
          \hbox to 1.05\wd1{d\kern-0.4ex{\char039}\hss}}}

\newcommand{\vla}[1]{{\color{red}{#1}}}
\newcommand{\bla}[1]{{\color{blue}{#1}}}

\newcommand{\mbf}[1]{\mathbf{#1}}

\newcommand*\xbar[1]{%
  \hbox{%
    \vbox{%
      \hrule height 0.5pt 
      \kern0.4ex
      \hbox{%
        \kern-0.05em
        \ensuremath{#1}%
        \kern-0.00em
      }%
    }%
  }%
}

\newsiamremark{rmk}{Remark}
\newsiamthm{thm}{Theorem}
\newsiamremark{hypothesis}{Hypothesis}
\crefname{hypothesis}{Hypothesis}{Hypotheses}
\newsiamthm{claim}{Claim}
\newsiamthm{prop}{Proposition}

\headers{A Well-Balanced PAMPA Method for Ripa System}{Y. Liu}

\graphicspath{{figures/}{./}}

\title{A Well-balanced Point-Average-Moment PolynomiAl-interpreted (PAMPA) Method for Shallow Water Equations \vla{with Horizontal Temperature Gradients} on Triangular Meshes\thanks{Submitted to the editors DATE.
\funding{The work of Y. Liu was supported by UZH Postdoc Grant, grant no. K-71120-05-01 SNSF grant 200020$\_$204917.}}}

\author{Yongle Liu\thanks{Institute of Mathematics, University of Z\"{u}rich, 8057 Z\"{u}rich, Switzerland
  (\email{yongle.liu@math.uzh.ch; liuyl2017@mail.sustech.edu.cn}).}}

\usepackage{amsopn}
\DeclareMathOperator{\diag}{diag}
